\documentclass[12pt,reqno]{amsart}

\usepackage{amssymb}
\usepackage{amscd}
\usepackage{amsfonts}
\usepackage{setspace}
\usepackage{version}
\usepackage{mathrsfs}
\usepackage{mathtools}
\usepackage{bm}

\usepackage{graphicx}

\usepackage{tikz}

\newtheorem{theorem}{Theorem}[section]
\newtheorem{lemma}[theorem]{Lemma}
\newtheorem{proposition}[theorem]{Proposition}
\newtheorem{corollary}[theorem]{Corollary}

\theoremstyle{definition}
\newtheorem{definition}{Definition}

\renewcommand{\leq}{\leqslant}
\renewcommand{\geq}{\geqslant}

\newcommand\T{\mathbf{T}}
\newcommand\Supp{\operatorname{Supp}}

\newcommand\main{\operatorname{main}}
\newcommand\error{\operatorname{error}}

\newcommand\tor{\operatorname{tor}}
\newcommand\sml{\operatorname{sml}}
\newcommand\unf{\operatorname{unf}}

\def\R{\mathbf{R}}

\def\Z{\mathbf{Z}}
\def\E{\mathbf{E}}
\def\P{\mathsf{P}}

\def\N{\mathbf{N}}

\def\eps{\varepsilon}

\newcommand{\md}[1]{\ensuremath{(\operatorname{mod}\, #1)}}
\newcommand{\mdsub}[1]{\ensuremath{(\mbox{\scriptsize mod}\, #1)}}
\newcommand{\mdlem}[1]{\ensuremath{(\mbox{\textup{mod}}\, #1)}}

\numberwithin{equation}{section}

\begin{document}

\title[Monochromatic solutions to $x + y = z^2$]{Monochromatic solutions to $x + y = z^2$}



\author{Ben Green}
\address{Mathematical Institute, Radcliffe Observatory Quarter, Woodstock Rd, Oxford OX2 6GG}
\email{ben.green@maths.ox.ac.uk}

\author{Sofia Lindqvist}
\address{Mathematical Institute, Radcliffe Observatory Quarter, Woodstock Rd, Oxford OX2 6GG}
\email{lindqvist@maths.ox.ac.uk}
\thanks{This work was partially supported by a grant from the Simons Foundation (award number 376201 to Ben Green) , and the first author is supported by ERC Advanced Grant AAS 279438. We thank both organisations for their support.}

\begin{abstract}
Suppose that $\N$ is $2$-coloured. Then there are infinitely many monochromatic solutions to $x + y = z^2$. On the other hand, there is a $3$-colouring of $\N$ with only finitely many monochromatic solutions to this equation.
\end{abstract}
\maketitle

\section{Introduction}

In this paper we will be concerned with the Ramsey theory of the equation $x + y = z^2$. It was shown relatively recently by Csikv\'ari, Gyarmati and S\'ark\"ozy \cite{gcs} that this equation is \emph{not} partition regular. Indeed, a 16-colouring of $\N$ is exhibited with no monochromatic solutions to $x + y = z^2$ other than the trivial one $x = y = z = 2$. There remains the question of whether the 16 here is optimal. Our main theorem completely answers this question.

\begin{theorem}\label{mainthm} There is a 3-colouring of $\N$ with no monochromatic solution to $x + y = z^2$ other than the trivial one. On the other hand, every 2-colouring of $\N$ has infinitely many monochromatic solutions to $x + y = z^2$.
\end{theorem}

The proof of the first statement is rather simple. It is given in Section \ref{sec2}. By contrast, the proof that every 2-colouring has infinitely many monochromatic solutions to $x + y = z^2$ is complicated and involves a surprisingly large number of tools from additive combinatorics and number theory. It occupies the remaining sections of the paper. We outline the argument now.

If $\N = V \cup W$ then let us assume that there are infinitely many $N$ such that $|V \cap [N,2N)| \geq N/2$. If this is not the case then a corresponding statement holds for $W$ and we may switch the roles of $V$ and $W$ in what follows. Suppose that there are no solutions to $x + y = z^2$ in either $V$ or $W$. By a fairly elaborate sequence of arguments involving the arithmetic regularity lemma  as well as certain Fourier-analytic and diophantine arguments, as well as a deep result of Lagarias, Odlyzko and Sloane, we use this to show that for some $q \in N$ and $c > 0$ the set $W$ contains the progression $\P([1, 1+c]; M, q) := \{ n \in \Z : M \leq n \leq (1 + c)M, n \equiv 0 \md{q}\}$ for infinitely many integers $M$. The details of these arguments may be found in Sections \ref{sec4} and \ref{sec5}, certain preliminary results having been assembled in Section \ref{sec3}. The proof is concluded in Section \ref{sec7} by performing an iterative argument to get a collection of further progressions inside $W$, eventually showing that all sufficiently large multiples of $q$ lie in $W$. An important ingredient here is a result concerning gaps between sums of two squares with certain constraints, proven in Section \ref{sec6}.

The fact that all sufficiently large multiples of $q$ lie in $W$ leads immediately to a contradiction, since $W$ then obviously contains infinitely many solutions to $x + y = z^2$.

We make heavy use of smooth cutoff functions in the latter half of the paper. The properties and constructions of these are recalled in Appendix \ref{bump-appendix}.

We remark that our arguments in fact give the following, logically stronger, result: if $N$ is large then any $2$-colouring of $[N, CN^8]$ has a monochromatic solution to $x + y = z^2$. Here $C$ is an absolute constant which could be computed in principle, but which would be astronomically large due to the application of the regularity lemma. We have found it easier to write the paper in such a way that this result does not immediately follow from our arguments as written, and we leave the interested reader to verify this statement.

Let us remark on the nice work of Khalfallah and Szemer\'edi \cite{sk} which, despite its rather similar title, concerns a somewhat different problem. They show that any finite colouring of $\N$ contains a solution to $x + y =z^2$ with $x$ and $y$ having the same colour (but not necessarily $z$).

We also remark that for the modular version of the problem the answer is very different. Indeed, the second author \cite{lindqvist} has shown that if $p > p_0(k)$ is a prime and if $\Z/p\Z$ is $k$-coloured, then there are $\gg_k p^2$ monochromatic solutions to $x + y = z^2$. 

\emph{Notation.} We collect here some notation used in the paper. Most of it is standard. If $X$ is a finite set then $\E_{x \in X}$ means $\frac{1}{|X|}\sum_{x \in X}$. For $t \in \R$, we write $e(t) := e^{2\pi i t}$. We write $\T = \R/\Z$ and $\T^d = (\R/\Z)^d$. We define a ``norm'' $\Vert \cdot \Vert_{\T^d} : \T^d \rightarrow [0,\frac{1}{2}]$ by defining $\Vert x \Vert_{\T^d} = \Vert \tilde x \Vert_{\ell^{\infty}(\R^d)}$, where $\tilde x$ is the unique element of $(-\frac{1}{2}, \frac{1}{2}]^d$ which projects to $x$ under the natural homomorphism from $\R^d$ to $\T^d$. 
The notation $X = O(Y)$ and $X \ll Y$ both mean that $X \leq CY$ for some constant $C$. Unless dependence on other parameters is indicated explicitly (for example $X \ll_{\eps} Y$), $C$ will be an absolute constant. 

The notation $\widehat{f}$ always denotes Fourier transform. At various points in the paper $f$ may be a function on $\Z$, $\R$ or $\T^d$. The definitions we are using are recalled in the text when there is any danger of confusion.

It is convenient to introduce a piece of notation which is less standard, but very useful. If $\Lambda \subset \N$ is a set of integers then we write $\sqrt{\Lambda} := \{n \in \N : n^2 \in \Lambda\}$ (this is \emph{not} the same as $\{\sqrt{n} : n \in \Lambda\}$).If $A \subset \N$ is a set, we write $2A = A + A := \{a + a' : a, a' \in A\}$. We will sometimes use notation such as $2\sqrt{2A}$, which means $\sqrt{A + A} + \sqrt{A + A}$.

Finally, as hinted above, when $I \subset \R$ is a closed interval we write $\P(I; N,q) := \{n \in \Z : \frac{n}{N} \in I, q | n\}$.

\section{A 3-colouring}\label{sec2}
In this short section we establish the easy part of Theorem \ref{mainthm}.  That is, we exhibit a 3-colouring of $\N$ for which the only monochromatic solution to $x + y = z^2$ is the trivial solution $x = y = z = 2$. We colour all the points in each dyadic block
\[A_i = \{n\in \N: 2^i\leq n< 2^{i+1}\}, \]
$i = 0,1,2,\dots$, in one colour $c_i$. We assign $c_0, c_1,c_2$ to be distinct, and then assign the colours $c_i$, $i \geq 3$, inductively in such a way that $c_i \notin \{c_{\lfloor i/2\rfloor}, c_{\lfloor i/2 \rfloor + 1}\}$. Note that this is possible since $\lfloor i/2 \rfloor + 1 < i$ for $i \geq 3$. 

Assume now that $x,y,z \in \N$ have the same colour and that $x + y = z^2$. Without loss of generality we may assume that $x \leq y$. Let $i \in \{0,1,2,\dots\}$ be such that $y \in A_i$.  Then $2^i < x+y< 2^{i+2}$, and hence $2^{i/2} < z < 2^{(i + 2)/2}$. Since $i/2 \geq \lfloor i/2\rfloor$ and $(i+2)/2 \leq \lfloor i/2\rfloor + 2$, it follows that  $z\in A_{[i/2]} \cup A_{[i/2]+1}$. By construction, the only way that such a $z$ can have the same colour as $y$ is if $i \in \{0,1,2\}$, in which case $x \leq y < 8$, and so $z = 2$ or $3$. An easy case check confirms that $x = y = z = 2$.

\section{Results from the literature}\label{sec3}

The rest of the paper is devoted to the harder part of Theorem \ref{mainthm}. In this section we assemble some basic ingredients from the literature.

We will need a version of Weyl's inequality, which gives a bound for exponential sums $\sum_{n \leq N} e(p(n))$ with $p : \N \rightarrow \R$ a polynomial. The usual proof of Weyl's inequality leads to a factor of $N^{o(1)}$ which renders the result worse than trivial in certain circumstances (the ``major arcs''). This is of no consequence in typical applications, which concern minor arc estimates in Waring's problem. Here, however, it is important to have an ``$\eps$-free'' result. Such results are well-known to experts, but it is hard to locate a convenient reference. Wooley \cite{wooley-free} discusses the pure power case (that is, sums of the form $\sum_{n \leq N} e(\alpha n^k)$), and it is likely that the same methods apply in greater generality, though the verification of this would involve a foray into the inner workings of \cite[Chapter 4]{vaughan-book}. 

A self-contained source for the purposes of this paper is \cite[Lemma 4.4]{green-tao-nilratner} (described in that paper as a ``reformulation'' of Weyl's inequality, a slightly inaccurate statement). Here is the statement.
\begin{proposition}\label{weyl}
Let $k \in \N$. Then there is a constant $C_k$ such that the following is true. Let $0 < \delta < 1/2$. Let $g : \Z \rightarrow \R$ be a polynomial of degree $k$ with leading coefficient $\alpha_k$ \textup{(}that is, $g(n) = \alpha_k n^k + \dots$\textup{)}. Suppose that $|\E_{n \in I} e(g(n))| \geq \delta$, where $I \subset \Z$ is a discrete interval. Then there is some $q \in \N$, $q \leq \delta^{-C_k}$, such that $\Vert q \alpha_k \Vert_{\R/\Z} \leq \delta^{-C_k}|I|^{-k}$.\end{proposition}

We will need this result in the cases $k = 2$ and $k = 4$. The proof in the latter case is essentially as hard as that of the general case. We remark that in Lemma \cite[Lemma 4.4]{green-tao-nilratner} the result is stated with $I = [N]$, but the general case follows trivially from this by translation (which does not affect the leading coefficient $\alpha_k$).

The following definition is relevant to much of the paper.

\begin{definition}\label{irrational-def}
Suppose that $\theta \in \R^d$. Let $N \geq 1$ be an integer and let $A > 0$ be some real parameter. We say that $\theta$ is $(A, N)$-irrational if whenever $\mathbf{r} \in \Z^d \setminus \{0\}$ and $\Vert \mathbf{r} \Vert_1 \leq A$ we have $\| \mathbf{r} \cdot \theta \|_{\T} \geq A/N$.
\end{definition}

We record a corollary of Proposition \ref{weyl}, phrased in the language of this definition. This corollary is the variant of Weyl's inequality that we have found to be most useful in this paper. 

\begin{corollary}\label{cor-weyl-1}
Let $k, N \in \N$. Suppose that $I \subset \Z$ is a \textup{(}discrete\textup{)} interval of length $\leq N^{1/k}$. Suppose that $\theta \in \R^d$ is $(A, N)$-irrational, and suppose that $\mathbf{r} \in \Z^d \setminus \{0\}$. Then
\[ |\sum_{n \in I} e(\mathbf{r} \cdot \theta n^k + \ldots)| \leq N^{1/k} \Vert \mathbf{r} \Vert_1 A^{-1/C_k}.\]  Here, $\ldots$ denotes polynomial terms in $n$ of degree $k-1$ or lower, and the estimate is uniform in the choice of these terms. 
\end{corollary} 
\begin{proof}
Suppose that the sum is $\geq \delta |I|$. Then, by Proposition \ref{weyl} there is some $q \in \N$, $q \leq \delta^{-C_k}$, such that $\Vert q \mathbf{r} \cdot \theta\Vert_{\R/\Z} \leq \delta^{-C_k} |I|^{-k}$. Since $\theta$ is $(A, N)$-irrational, we have either (1) $q\Vert \mathbf{r} \Vert_1\geq A$ or (2) $\delta^{-C_k} |I|^{-k} \geq A/N$. In case (1), the bound on $q$ implies that $\delta^{-C_k} \Vert \mathbf{r} \Vert_1 \geq A$. In case (2), we have $\delta^{-C_k} \geq A$. Hence in either case we have $\delta^{-C_k} \Vert \mathbf{r} \Vert_1 \geq A$, and hence $\delta \leq (\Vert \mathbf{r} \Vert_1 /A)^{1/C_k}$. The result follows (in fact with $\Vert \mathbf{r} \Vert_1$ replaced by the smaller quantity $\Vert \mathbf{r} \Vert_1^{1/C_k}$).
\end{proof}

Turning to a different type of ingredient of the paper, we require the following estimate.

\begin{proposition}\label{ell-6}
Let $S \subset \{1,\dots, N\}$ be any set of squares. For $t \in \R/\Z$, write $\widehat{1}_S(t) := \sum_{n \in S} e(t n)$. Then $\int^1_0 |\widehat{1}_S(t)|^6 dt \ll N^2$.
\end{proposition}
\begin{proof}
It is easy to see that the integral is $\sum_{x \leq 3N} r_{3,S}(x)^2$, where $r_{3,S}(x)$ is the number of ways of writing $x$ as $n_1 + n_2 + n_3$ with $n_1, n_2, n_3 \in S$. This quantity is obviously largest when $S$ is the set of \emph{all} squares $\leq N$. In this case, the stated bound is a well-known consequence of the Hardy-Littlewood method.
\end{proof}
\emph{Remark.} Using more advanced methods of harmonic analysis (related to the Tomas-Stein restriction theorem) one can show a bound $\int^1_0 |\widehat{1}_S(t)|^q \ll_q N^{q/2 - 1}$ for any $q > 4$.\vspace{8pt}

Finally, we will also use the following result of Lagarias, Odlyzko and Shearer \cite{los1}. 

\begin{proposition}\label{sumset-residue}
Suppose that $S \subset \Z/q\Z$, where $q$ is a positive integer, and that $|S| > \frac{11}{32}q$. Then $S + S$ contains a quadratic residue modulo $q$.
\end{proposition}
\emph{Remarks.} The $\frac{11}{32}$ in this theorem is sharp. For our purposes, $\frac{11}{32}$ could be replaced by any constant less than $\frac{1}{2}$. A simpler proof of such a statement could probably be extracted from \cite{los1} or the companion paper \cite{los2}, but we do not know of any argument that could be described as in any way routine. 

Instead of the result of Lagarias, Odlyzko and Shearer, it would suffice to have the following statement: there is some $\eta_k > 0$ such that if $(1 - \eta_k) q$ of the elements of $\Z/q\Z$ are $k$-coloured then there are $x, y$ of the same colour with $x + y$ a square. We believe that such a statement can be established relatively painlessly using a simplified version of the arguments of Khalfallah and Szemer\'edi \cite{sk}. The second author provides an account of this in an unpublished note \cite[Theorem 1.2]{lindqvist-unpub}.

\section{Capturing most of the squares in a Bohr set}
\label{sec4}
This section contains the technical heart of the paper. Our aim is to prove the following result. Here, and in what follows, $\mathfrak{S}(b,q)$ denotes the number of solutions to $x^2 \equiv b \md{q}$ with $x \in \Z/q\Z$.

\begin{proposition}\label{sec-3-main}
Let $\eta > 0$, and let $\Omega : \N^3 \rightarrow \N$ be a function \textup{(}which may depend on $\eta$\textup{)}, nondecreasing in each variable. Suppose that $N > N_0(\Omega, \eta)$ is sufficiently large, and let $A \subset [N,2N]$ be a set of size at least $N/2$. Then there are $q, d = O_{\eta,\Omega}(1)$, $\eps \gg_{\eta,\Omega} 1$, $b \in \Z/q\Z$, $x \in [2,4]$, and $\theta, z \in  \R^d$ such that 
\begin{enumerate}
\item $b$ is a quadratic residue modulo $q$;
\item $\theta$ is $(\Omega(q,d,1/\eps),N)$-irrational;
\item $A + A$ contains all but at most $\eta \mathfrak{S}(b,q)(2\eps)^{d+1}q^{-1} N^{1/2}$ of the squares in the set $\{n \in \N : n \equiv b \md{q}, |\frac{n}{N} - x| , \Vert \theta n - z\Vert_{\T^d} \leq \eps\}$.
\end{enumerate}
\end{proposition}

\emph{Remarks.} The assumption that $|A| \geq N/2$ could be weakened to $|A| \geq cN$ for any $c > 11/32$, using essentially the same proof. We do not record this explicitly as Proposition \ref{sec-3-main} seems unlikely to be of independent interest. In our applications, $\eta$ will be an absolute constant which could be specified explicitly if desired ($\eta = 10^{-10}$ should certainly be admissible). 

The key tool in the proof of Proposition \ref{sec-3-main} will be the \emph{arithmetic regularity lemma}, introduced in \cite{green-arith-regularity}. The formulation we use here, in a more general guise, is the main result of \cite{green-tao-arithregularity}. That paper is long and quite difficult, but only Sections 1 and 2 of it are relevant to us. Furthermore, that paper establishes a regularity lemma for the Gowers $U^{s+1}$-norm for general $s$, whereas we only need the case $s = 1$. This means that the notion of a \emph{nilsequence}, beyond the abelian case, is not relevant here. A complete, self-contained proof of the arithmetic regularity lemma in the form we need it here can be written up in less than 10 pages. Conveniently, such a writeup has been provided by Sean Eberhard \cite{sean-writeup}.

Here is the arithmetic regularity lemma in the form we will need it.

\begin{proposition}\label{ar} Suppose we are given $\delta > 0$ and an increasing function $\mathcal{F} : \N \rightarrow \R_+$. Then there exists $M_{\textup{max}} \ll_{\delta,\mathcal{F}} 1$ such that for any function $f : [N,\dots, 2N) \to [0,1]$ there is an $M\leq M_{\max}$  and a decomposition $f = f_{\tor} + f_{\sml}+ f_{\unf}$ into functions taking values in $[-1,1]$,  where $\sum_{N \leq n < 2N} |f_{\sml}(n)| \leq \delta N$, $\| \widehat{f}_{\unf} \|_{\infty} \leq N/\mathcal{F}(M)$ and $f_{\tor}(n) = F(n \mdlem{q}, n/N, \theta n)$ for some $q,d \leq M$ and some function $F : \Z/q\Z \times [1,2] \times \T^d \rightarrow [0,1]$ with Lipschitz constant at most $M$. Furthermore $\theta$ may be taken to be $(\mathcal{F}(M), N)$-irrational.
\end{proposition}

We remark that in the works previously cited the function $f_{\unf}$ was controlled in terms of the Gowers $U^2$-norm, rather than in terms of the supremum norm of the Fourier transform, defined by
\[ \widehat{f}_{\unf}(t) := \sum_{N \leq n < 2N} f_{\unf}(n) e(-tn),\] where $e(x) = e^{2\pi i x}$. However it is well-known (and easy to prove) that for bounded functions these norms are essentially equivalent.

Moreover $f_{\sml}$ is traditionally controlled in the $\ell^2$-norm, rather than the $\ell^1$-norm as we have here. However, since $f_{\sml}$ is bounded by $1$, these two norms are equivalent too. Thus Proposition \ref{ar} is equivalent to the arithmetic regularity lemma as usually stated. 

Let us now begin the proof of Proposition \ref{sec-3-main} in earnest. Apply Proposition \ref{ar} with $f = 1_A$, $\delta < \eta$ some small constant ($\delta = 10^{-100}$ would be permissible), and the function $\mathcal{F}$ to be specified later (it will depend on $\Omega$ and $\eta$). This gives integers $q, d \leq M$, $\theta \in \R^d$ and $F : \Z/q\Z \times [1,2] \times \T^d \rightarrow [0,1]$ and a decomposition
\begin{equation}\label{ar-decomp} 1_A = f_{\tor} + f_{\sml} + f_{\unf}\end{equation} with the properties described in the statement of Proposition \ref{ar} just given.

\begin{lemma}\label{counting}
Suppose that $\delta$ is sufficiently small and that $\mathcal{F}$ grows sufficiently rapidly. Then $\int F d\mu > \frac{9}{20}$, where $\mu$ denotes the natural\footnote{The product of the uniform probability measure on $\Z/q\Z$, Lebesgue measure on $\R$ and normalised Lebesgue measure on $\T^d$.} measure on $\Z/q\Z \times \R \times \T^d$.
\end{lemma}
\emph{Remark.} Here, $\frac{9}{20}$ is simply a convenient fraction less than $\frac{1}{2}$. In fact, $\int F d\mu$ can be made as close to $\frac{1}{2}$ as one wishes by reducing $\delta$ and increasing $\mathcal{F}(M)$.

\begin{proof}
We begin by noting that, by assumption,
\begin{equation}\label{eq-11} \E_{N \leq n < 2N} 1_A(n) \geq \frac{1}{2}.\end{equation}
If $\delta < \frac{1}{100}$ then 
\begin{equation}\label{eq-12} |\E_{N \leq n < 2N} f_{\sml}(n)| < \frac{1}{100}.\end{equation}
Also, introducing a smooth majorant $\psi$ for $[N, 2N)$ with $\psi(n) = 1$ for $N \leq n < 2N$ we have
\begin{align*} |\E_{N \leq n < 2N} f_{\unf}(n)|  & = |\frac{1}{N} \sum_n \psi(n) f_{\unf}(n) | \\ &  = |\frac{1}{N} \int^1_0 \widehat{\psi}(t) \widehat{f_{\unf}}(t) dt| \\ & \leq \frac{\Vert \widehat{\psi} \Vert_1}{\mathcal{F}(M)}.\end{align*}
With an appropriate choice of $\psi$ (see Lemma \ref{app-1} for details) we have $\Vert \widehat{\psi} \Vert_1 = O(1)$, and so if $\mathcal{F}(M)$ is sufficiently large it follows that 
\begin{equation}\label{eq-13} |\E_{N \leq n < 2N} f_{\unf}(n)|< \frac{1}{100}.\end{equation}
We also have
\[ \E_{N \leq n < 2N} f_{\tor}(n) = \E_{N \leq n < 2N} F(n \md{q}, \frac{n}{N}, \theta n).\] 
However, it was proven\footnote{This is not an especially difficult argument: roughly, one approximates $F$ by a function with finite Fourier support, then uses the irrationality of $\theta$ in estimating the resulting exponential sums.} in \cite[Lemma A.4]{eberhard-green-manners}
that, if $\mathcal{F}$ grows sufficiently rapidly and if $N$ is big enough, 
\begin{equation}\label{eq-14} |\E_{N \leq n < 2N} F(n \md{q}, \frac{n}{N}, \theta n) - \int F d\mu| < \frac{1}{100}.\end{equation}
Combining \eqref{eq-11}, \eqref{eq-12}, \eqref{eq-13}, \eqref{eq-14} concludes the proof.
\end{proof}

Now let $U \subset \Z/q\Z$ be the set of all $u \in \Z/q\Z$ for which 
\begin{equation}\label{cond-1} \int^2_1 \int_{\T^d} F(u, x, z) dz dx \geq \frac{1}{20}\end{equation} and for which
\begin{equation}\label{cond-2} \sum_{\substack{N \leq n < 2N \\ n \equiv u \mdsub{q}}} |f_{\sml}(n)| \leq \frac{20\delta}{q}N.\end{equation}
One should think, informally, of these being the residue classes $\md{q}$ on which $A$ has ``significant mass''.

\begin{lemma}\label{lem1-a}
Suppose that $\delta$ is sufficiently small and that $\mathcal{F}$ grows sufficiently rapidly. There are elements $u, u' \in U$ such that $u + u'$ is a quadratic residue modulo $q$.
\end{lemma}
\begin{proof}
Let $U_1 \subset \Z/q\Z$ be the set of all $u$ for which \eqref{cond-1} fails, and $U_2$ the set of all $u$ for which \eqref{cond-2} fails. Since $\sum_{N \leq n < 2N} |f_{\sml}(n)| \leq \delta N$, we have
\[ |U_2| \leq  \frac{q}{20} .\]

Furthermore by Lemma \ref{counting} we have
\begin{align*}
\frac{9}{20} < \int F d\mu  & = \frac{1}{q}\sum_{u \in \Z/q\Z} \int^2_1 \int_{\T^d} F(u, x, z) dz dx \\ & \leq \frac{1}{20} + \frac{1}{q}|(\Z/q\Z) \setminus U_1|. 
\end{align*}
It follows that 
\[ |U| \geq |(\Z/q\Z) \setminus U_1| - |U_2| \geq (\frac{9}{20} - \frac{1}{20} - \frac{1}{20}) q > \frac{11q}{32}.\]
The result now follows from Proposition \ref{sumset-residue}.
\end{proof}

Henceforth, we will fix two residue classes $u, u' \in U$ for which $u + u'$ is a quadratic residue modulo $q$. Define parameters $\eps > \eps' > 0$ by
\begin{equation}\label{eps-def} \eps := \frac{\delta}{M} \end{equation} and
\begin{equation}\label{eps-def-prime} \eps' := \frac{\delta}{dq}(2\eps)^{d+1}.
\end{equation}
Note that since $q, d \leq M$ we have
\begin{equation}\label{eps-prime-below} \eps' \geq \frac{\delta}{M^2}(\frac{2\delta}{M})^{M + 1} \gg_{\delta, M} 1.\end{equation} (The precise form of this bound is unimportant; what matters is that there is a lower bound depending only on $\delta$ and $M$.)

For $x, x' \in [1,2]$ and $z,z' \in \T^d$, define
\[ E_{x, z} := \sum_{\substack{N \leq n < 2N \\ n \equiv u \mdsub{q} \\ |\frac{n}{N} - x| \leq \eps  \\  \Vert \theta n - z \Vert_{\T^d} \leq \eps }} |f_{\sml}(n)| \qquad \mbox{and} \qquad  E'_{x', z'} := \sum_{\substack{N \leq n < 2N \\ n \equiv u' \mdsub{q} \\ |\frac{n}{N} - x'| \leq \eps' \\   \Vert \theta n - z'  \Vert_{\T^d} \leq \eps' }} |f_{\sml}(n)|.\] 
We have
\begin{align*} \int^2_1 \int_{\T^d} E_{x, z}dz dx & = \sum_{\substack{N \leq n < 2N \\ n \equiv u \mdsub{q}}} |f_{\sml}(n)| \int^2_11_{|\frac{n}{N} - x| \leq \eps } dx \int_{\T^d}1_{\Vert \theta n - z \Vert_{\T^d} \leq \eps } dz \\ & \leq (2\eps)^{d+1}\sum_{\substack{N \leq n < 2N \\ n \equiv u \mdsub{q}}} |f_{\sml}(n)| \leq (2\eps)^{d+1} \frac{20\delta}{q} N, \end{align*}
the last step being a consequence of \eqref{cond-2}. It follows from this and \eqref{cond-1} that
\[ \int^2_1 \int_{\T^d} \bigg( F(u,x,z) - \frac{q}{800 N \delta (2\eps)^{d+1}} E_{x,z}   \bigg)dz dx \geq \frac{1}{40},\] and so there are specific choices of $x, z$ such that 
\[ F(u,x,z) - \frac{q}{800 N \delta (2\eps)^{d+1}} E_{x,z}  \geq \frac{1}{40},\] which implies that 
\begin{equation}\label{xz-cond} F(u,x,z) \geq \frac{1}{40} \qquad \mbox{and} \qquad E_{x,z} \leq \frac{800 \delta N}{q} (2\eps)^{d+1}.\end{equation}
Similarly, there are $x', z'$ such that 
\begin{equation}\label{xz-cond-prime}  F(u',x',z') \geq \frac{1}{40} \qquad \mbox{and} \qquad E_{x',z'} \leq \frac{800 \delta N}{q} (2\eps')^{d+1}. \end{equation}

From now on, we fix these specific choices of $x, z, x', z'$ and set
\begin{equation}\label{x-def} X := \{ n \in \N : n \equiv u \md{q}, |\frac{n}{N} - x| , \|\theta n - z\|_{\T^d} \leq \eps \},\end{equation}
\begin{equation}\label{x-prime-def}  X' := \{ n \in \N : n \equiv u' \md{q}, |\frac{n}{N} - x'| , \|\theta n - z'\|_{\T^d} \leq \eps' \},\end{equation} and
\begin{equation} \label{y-def} Y :=  \{ n \in \N : n \equiv u + u' \md{q},  |\frac{n}{N} - (x + x')|  ,  \|\theta n - (z + z')\|_{\T^d} \leq \eps\}.
\end{equation}

Note that with this notation \eqref{xz-cond}, \eqref{xz-cond-prime} imply
\begin{equation}\label{unsmoothed} 
\sum_{n \in X} |f_{\sml}(n)| \ll \delta (2\eps)^{d+1} q^{-1} N,\qquad \sum_{n \in X'} |f_{\sml}(n)| \ll \delta (2\eps')^{d+1}q^{-1} N . \end{equation}

\begin{lemma}\label{lem-square}
Suppose that $\mathcal{F}$ grows sufficiently rapidly, and that $N$ is sufficiently large in terms of $\delta, M$. Then the number of squares in $Y$ is $\ll (2\eps)^{d+1} q^{-1}\mathfrak{S}(u + u', q) N^{1/2}$.
\end{lemma}
\begin{proof}
Let $\mathscr{A}$ be the set of all $a \in \Z/q\Z$ for which $a^2 \equiv u + u' \md{q}$. Thus $|\mathscr{A}| = \mathfrak{S}(u + u', q)$. An upper bound for the number of squares in $Y$ is then 
\[ \sum_{a \in \mathscr{A}}\sum_{n \in I} 1_{n \equiv a \mdsub{q}} \psi^+_{\eps}(\theta n^2 - z - z'),\]
where $I = [(x + x' - \eps)^{1/2} N^{1/2},(x + x' + \eps)^{1/2} N^{1/2}]$ and $\psi^+_{\eps}$ is the majorant for the characteristic function of the ball $B_{\eps}(0)$ in $\T^d$ constructed in Lemma \ref{lem-a2}. 
Fourier expanding
\[ 1_{n \equiv  a  \md{q}} = \frac{1}{q} \sum_{r \mdsub{q}} e(-\frac{ra}{q}) e(\frac{rn}{q} )\] and
\[ \psi^+_{\eps}(t) = \sum_{\mathbf{r} \in \Z^d} \widehat{\psi^+_{\eps}}(\mathbf{r}) e(\mathbf{r} \cdot t),\] this may be written as
\begin{equation}\label{star-eq-1} \sum_{a \in \mathscr{A}} \frac{1}{q} \sum_{r \mdsub{q}} e(-\frac{ra}{q})\sum_{\mathbf{r} \in \Z^d}\widehat{\psi^+_{\eps}}(\mathbf{r}) e(-\mathbf{r} \cdot (z + z')) \sum_{n \in I} e(\mathbf{r} \cdot \theta n^2 + \frac{rn}{q} ).\end{equation} The contribution from $\mathbf{r} = 0$ is 
\[ \frac{1}{q}(\int \psi_{\eps}^+ )\sum_{a \in \mathscr{A}} \sum_{r \mdsub{q}} e(-\frac{ra}{q}) \sum_{n \in I} e(\frac{rn}{q}).\] 
If $r \neq 0$, the inner sum over $n$ is at most $q$ in magnitude, since the sum of $e(rn/q)$ over any interval of length $q$ is zero. The total contribution from these terms is thus bounded independently of $N$, and so may be ignored if $N$ is large enough. The contribution from $r = 0$ is $\frac{1}{q}\mathfrak{S}(u + u', q) (\int \psi_{\eps}^+) |I|$, which is $\ll (2\eps)^{d+1}q^{-1}\mathfrak{S}(u + u', q) N^{1/2}$ by Lemma \ref{lem-a2} (1) and the bound $|I| \ll \eps N^{1/2}$. 
 The contribution to \eqref{star-eq-1} from $\mathbf{r} \neq 0$ is bounded above by
\[ \mathfrak{S}(u + u', q) \sum_{\mathbf{r} \in \Z^d \setminus \{0\}} |\widehat{\psi^+_{\eps}}(\mathbf{r})| \sup_{r \mdsub{q}} \big|\sum_{n \in I} e(\mathbf{r} \cdot \theta n^2 + \frac{r}{q}n )   \big|.\]
By Corollary \ref{cor-weyl-1} and Lemma \ref{lem-a2} (2), this is
\[ \ll qN^{1/2}\mathcal{F}(M)^{-1/C_2} \sum_{\mathbf{r} \in \Z^d \setminus \{0\}} |\widehat{\psi^+_{\eps}}(\mathbf{r})| \Vert \mathbf{r} \Vert_{1} \ll_{\delta, M} N^{1/2}\mathcal{F}(M)^{-1/C_2} .\]
(Lemma \ref{lem-a2} (2) gives an implied constant depending on $d, \eps$, but we have $d \leq M$ and $\eps = \delta/M$.)
Hence if $\mathcal{F}$ is chosen to be sufficiently rapidly-growing, this is smaller than $(\frac{2\delta}{M})^{M+1}M^{-1} N^{1/2}$, which is at most $N^{1/2} (2\eps)^{d+1} q^{-1} N^{1/2}$.
\end{proof}

We will also need the following fact, proven using very similar techniques.

\begin{lemma}\label{lem-square-2}
Suppose that $\mathcal{F}$ grows sufficiently rapidly, and that $N$ is sufficiently large in terms of $\delta, M$. Suppose that $n \in X$. Then the number of $n' \in X'$ for which $n + n'$ is a square is $\ll (2\eps')^{d+1}q^{-1}\mathfrak{S}(u + u', q) N^{1/2}$, uniformly in $n$.\end{lemma}
\begin{proof}
Once again, write $\mathscr{A}$ for the set of square roots of $u + u'$ in $\Z/q\Z$. Writing $m^2 = n + n'$, an upper bound for the quantity in question is
\[ \sum_{a \in \mathscr{A}}\sum_{m \in J} 1_{m \equiv  a \mdsub{q}} \psi^+_{\eps'} (\theta m^2 - \theta n - \theta z'),\] where $J = [(n + (x' - \eps') N)^{1/2}, (n + (x' + \eps') N)^{1/2}]$ and $\psi^+_{\eps'}$ is the majorant constructed in Lemma \ref{lem-a2} (but now with the smaller parameter $\eps'$). Expanding in Fourier series much as before, this may be written as 
\[ \sum_{a \in \mathscr{A}}\frac{1}{q} \sum_{r \mdsub{q}} e(-\frac{ra}{q}) \sum_{\mathbf{r} \in \Z^d} \widehat{\psi^+_{\eps'}}(\mathbf{r}) e(-\mathbf{r} \cdot \theta (n + z')) \sum_{m \in J} e(\mathbf{r} \cdot \theta m^2 + \frac{r m}{q}). \]
Arguing in an essentially identical fashion to the proof of Lemma \ref{lem-square}, we see that this is bounded by a main term of size $\ll (2\eps')^{d+1} q^{-1}\mathfrak{S}(u + u', q) N^{1/2}$ plus an error of size $\ll_{\delta, M} N^{1/2} \mathcal{F}(M)^{-1/C_2}$. Choosing $\mathcal{F}$ to be sufficiently rapidly-growing, and recalling from \eqref{eps-prime-below} that $\eps' \gg_{\delta, M} 1$, this can be made $\ll (2\eps')^{d+1}q^{-1}\mathfrak{S}(u+u',q) N^{1/2}$.
\end{proof}

Finally, we need yet another fact with a similar proof. Define the set $Y_- \subset Y$ to be
\[\{ n \in \N : n \equiv u + u' \md{q},  |\frac{n}{N} - (x + x')| , \Vert \theta n - (z + z')\Vert_{\T^d} \leq \eps- 2\eps'\}.\label{y-minus-def}\]
\begin{lemma}\label{lem36}
Suppose that $\mathcal{F}$ grows sufficiently rapidly. Then the number of squares in $Y \setminus Y_{-}$ is $\ll \delta (2\eps)^{d+1} q^{-1} N^{1/2}$.
\end{lemma}
\begin{proof}
If $n \in Y \setminus Y_{-}$ then either
\begin{equation}\label{prob1} \eps - 2\eps' < |\frac{n}{N}  - (x + x')| < \eps \end{equation}
or
\begin{equation}\label{prob2}  \eps - 2\eps' < \Vert \theta_i n - (z_i + z'_i)\Vert_{\T^d} < \eps \end{equation} for some $i \in \{1,\dots, d\}$.
The number of squares satisfying \eqref{prob1} is elementarily seen to be $O(\eps' N^{1/2})$, which\footnote{Obviously this bound is rather crude, as we have completely ignored the fact that additionally $n \equiv u + u' \md{q}$ and $\Vert \theta n - (z + z') \Vert_{\T^d} \leq \eps$, but this is of little consequence in the grand scheme of the argument.} is bounded as desired because of the choice of $\eps'$ (cf. \eqref{eps-def-prime}). 

We now obtain an upper bound for the number of squares satisfying \eqref{prob2}. By translating the function $\psi^+_{\eps}$ constructed in Lemma \ref{lem-a2} (with $d = 1$ in that lemma) we may obtain a smooth majorant $\psi$ for the interval $\{ t \in \T : \eps - 2\eps' < \Vert t - (z_i + z'_i)\Vert_{\T}< \eps\}$ such that
\begin{equation}\label{psi-prop} \int \psi \ll \eps', \quad \sum_r |\widehat{\psi}(r)| |r| \ll_{\eps'} 1.\end{equation} Then the number of squares satisfying \eqref{prob2} is bounded above by
\[ \sum_{n \leq 2N^{1/2}} \psi(\theta_i n^2) = \sum_{r \in \Z} \widehat{\psi}(r) \sum_{n \leq 2N^{1/2}} e(r \theta_i n^2).\]
The term with $r = 0$ is $2N^{1/2} (\int \psi) \ll \eps' N^{1/2}$. By Corollary \ref{cor-weyl-1} (applied with $d = 1$) the contribution from the terms with $r \neq 0$ is \[ \ll  N^{1/2}\mathcal{F}(M)^{-1/C_2}\sum_{r \neq 0} |\widehat{\psi}(r)||r|. \]
By \eqref{psi-prop} this is $\ll_{\eps'} N^{1/2} \mathcal{F}(M)^{-1/C_2}$ which, in view of \eqref{eps-prime-below}, is $O(\eps' N^{1/2})$ provided $\mathcal{F}(M)$ grows sufficiently rapidly. Thus the total number of $n$ satisfying \eqref{prob2} for some $i \in \{1,\dots, d\}$ is $O(\eps' d N^{1/2})$, which is bounded as claimed by the choice of $\eps'$.
\end{proof}

To complete the proof of Proposition \ref{sec-3-main} it suffices to show that $A + A$ contains all but $\ll \delta (2\eps)^{d+1} q^{-1} \mathfrak{S}(u + u', q)N^{1/2}$ of the squares in $Y$. Indeed if $\delta$ is chosen small enough then this will be $\leq \eta (2\eps)^{d+1} q^{-1} \mathfrak{S}(u + u', q)N^{1/2}$, the bound claimed.
Let $S \subset Y$ be the set of all squares in $Y$ which are not in $A + A$; thus it suffices to establish the bound
\begin{equation}\label{s-bound} |S| \ll \delta (2\eps)^{d+1} q^{-1}\mathfrak{S}(u + u', q) N^{1/2} . \end{equation}
 Recall the definitions \eqref{x-def}, \eqref{x-prime-def} of $X, X'$. We will need to introduce smoothed approximants $\chi, \chi'$ to the characteristic functions of $X, X'$ respectively, with the following properties.

\begin{enumerate}
\item $\chi$ is a minorant for $X$, that is to say $0 \leq \chi(n) \leq 1_X(n)$ for all $n$;
\item $\chi'$ is a minorant for $X'$, that is to say $0 \leq \chi'(n) \leq 1_X(n)$ for all $n$;
\item $\chi(n) = 1$ on the set $\{n \in \N : n \equiv u \md{q}, |\frac{n}{N} - x| , \Vert \theta n - z \Vert_{\T^d}\leq \eps - \eps'\}$;
\item $\int^1_0 |\widehat{\chi}(t)| dt, \int^1_0 |\widehat{\chi'}(t)|dt = O_M(1)$;
\item $\sum_n \chi'(n) \gg (2\eps')^{d+1}q^{-1} N$.
\end{enumerate}
Such a function is constructed in Lemma \ref{lem-a3} (which must be applied twice, once with parameter $\eps$ and once with parameter $\eps'$).

In particular it follows from \eqref{unsmoothed} that 
\begin{equation}\label{small-bds} \sum_n |f_{\sml}\chi(n)| \ll \delta (2\eps)^{d+1} q^{-1} N, \qquad \sum_n |f_{\sml}\chi'(n)| \ll \delta (2\eps')^{d+1} q^{-1} N.\end{equation}

Our assumption that $A + A$ is disjoint from $S$ implies that
\begin{equation}\label{to-bound} \sum_{n \in S} (1_{A}\chi \ast 1_{A}\chi')(n)  = 0.\end{equation}
To investigate this expression, we use the decomposition from the regularity lemma, 
\[ 1_A = f_{\tor} + f_{\sml} + f_{\unf}.\] The left-hand side of \eqref{to-bound} may then be expanded as a sum of 9 terms
\[ T_{\bullet, \bullet'} :=  \sum_{n \in S} (f_{\bullet}\chi \ast f_{\bullet'} \chi')(n) ,\] where $\bullet, \bullet' \in \{ \tor, \sml, \unf\}$. 
Thus
\begin{equation}\label{tors} |T_{\tor, \tor}| \leq \sum_{(\bullet, \bullet') \neq (\tor, \tor)} |T_{\bullet, \bullet'}|.\end{equation}
We analyse these 9 terms $T_{\bullet, \bullet'}$ separately, beginning with the ``main term'' $T_{\tor, \tor}$.

 Writing \[ f_{\tor}(n) = F(n \md{q}, \frac{n}{N}, \theta n),\] we may expand $T_{\tor, \tor}$ as  
\begin{align*}  \sum_{n \in S}&  \sum_m   F(m \md{q}, \frac{m}{N}
 , \theta m) \chi(m)   \\ & \times F(n - m \md{q}, \frac{n - m}{N}, \theta(n - m)) \chi'(n - m).\end{align*}
Since $\chi(m)$ is supported where $m \equiv u \md{q}$ and $|\frac{m}{N} - x| ,\|\theta m - z\|_{\T^d} \leq \eps$, and since $F$ is $M$-Lipschitz, we have using \eqref{xz-cond} that
\[ F(m \md{q}, \frac{m}{N}, \theta m)\chi(m) = (F(u, x, z) + O(M\eps))\chi(m) \geq \frac{1}{80}\chi(m)\] if $\delta$ is sufficiently small (note, recalling the definition \eqref{eps-def} of $\eps$, that $M\eps = \delta$). Similarly,
\[ F(n - m \md{q}, \frac{n - m}{N}, \theta(n - m))\chi'(n - m) \geq \frac{1}{80} \chi'(n- m).\]
It follows that 
 \begin{align} \nonumber T_{\tor, \tor} & \gg \sum_{n \in S} \sum_m \chi(m) \chi'(n - m)\\ & = \sum_{n \in S} \sum_m \chi(n - m) \chi'(m).\label{eq73}\end{align}
Recall the definition \eqref{y-minus-def} of $Y_- \subset Y$. If $n \in Y_-$ and $m \in \Supp(\chi') \subset X'$ then $n - m \equiv u \md{q}$ and $|\frac{n - m}{N} - x| , \|\theta(n - m) - z\|_{\T^d} \leq \eps - \eps'$, and therefore by property (3) of $\chi$ we have $\chi(n - m) = 1$. It follows from these observations, \eqref{eq73} and point (5) of the properties of $\chi, \chi'$ that 
 \begin{align}\nonumber T_{\tor, \tor} & \gg \sum_{n \in S \cap Y_{-}} \sum_m \chi'(n-m) \chi'(m) \\ & \nonumber \gg |S \cap Y_{-}| \sum_{m} \chi'(m) \\ & \gg |S \cap Y_-| (2 \eps')^{d+1}q^{-1} N. \label{tor-tor}\end{align}
We set this estimate aside for later use. 

Next we look at the terms $T_{\bullet, \bullet'}$ in which $\bullet' = \sml$. Here we require the \emph{a priori} bound
\begin{equation}\label{eq701} |S| \ll (2\eps)^{d+1} q^{-1} \mathfrak{S}(u + u', q) N^{1/2}.\end{equation}
This is, of course, weaker than the result we are trying to prove, but it follows immediately from Lemma \ref{lem-square}. All of these terms $T_{\bullet, \sml}$ have the form
\[ T_{\bullet, \sml} = \sum_{n \in S} (g \ast f_{\sml} \chi')(n) = \sum_{n \in S} \sum_m g(n - m) f_{\sml}\chi'(m),\] where $g$ is some function bounded pointwise by $1$.
Thus
\[ |T_{\bullet, \sml}| \leq |S| \sum_m |f_{\sml} \chi'(m)|\] and so, by \eqref{eq701} and \eqref{small-bds}, 
\begin{equation}\label{bullet-sml} T_{\bullet, \sml} \ll \delta  (4\eps \eps')^{d+1} q^{-2}\mathfrak{S}(u + u', q) N^{3/2}.\end{equation}

 Next we turn to the bounding of
\[ T_{\sml, \tor} = \sum_{n \in S} (f_{\sml} \chi \ast f_{\tor} \chi')(n).\] This expands as 
\[ \sum_{n \in S} \sum_m f_{\sml}\chi(n - m) F(m \md{q}, \frac{m}{N}, \theta m) \chi'(m).\]
By the Lipschitz property of $F$ and the fact that $\chi'$ is supported on $X'$, this is
\[ F(u', x', z')\sum_{\substack{m \\ n \in S}}f_{\sml} \chi(n - m) \chi'(m) + O(\eps' M) \sum_{\substack{m \\ n \in S}} |f_{\sml}\chi(n - m)| \chi'(m).\] Since $\eps' < \eps < 1/M$, it follows that 
\[ T_{\sml, \tor}  \ll \sum_{n \in S} \sum_m |f_{\sml}\chi(n - m)| \chi'(m)= \sum_{n', m} |f_{\sml}\chi(n')| \chi'(m) 1_S(n' + m).\] By \eqref{small-bds}, this is
\[ \ll \delta (2\eps)^{d+1} q^{-1} N^{1/2}\sup_{n' \in \Supp \chi} \sum_m \chi'(m) 1_S(n' + m).\]
By Lemma \ref{lem-square-2} and the fact that $\Supp \chi \subset X$, $\Supp \chi' \subset X'$, we conclude that 
\begin{equation}\label{sml-tor} T_{\sml, \tor} \ll \delta  (4\eps \eps')^{d+1} q^{-2}\mathfrak{S}(u + u', q) N^{3/2}.\end{equation}

 In all of the remaining terms $T_{\bullet, \bullet'}$ that we have yet to bound, at least one of $\bullet, \bullet'$ is $\unf$. If $\bullet = \unf$ then such a term has the form
 \[ T_{\unf, \bullet'} = \sum_{n \in S} (f_{\unf} \chi \ast g)(n),\] where $g$ is some function bounded pointwise by $1$.
 This may be written in Fourier space as \[\int^1_0 \widehat{f_{\unf} \chi}(t) \widehat{g}(t) \widehat{1}_{S}(t) dt,\] where $g$ is a bounded function. By H\"older's inequality, the right-hand side here is bounded above by
\begin{equation}\label{holdered} \Vert \widehat{f_{\unf} \chi} \Vert_{\infty}^{1/3} \big(\int^1_0 |\widehat{f_{\unf} \chi} |^2  \big)^{1/3} \big( \int^1_0 |\widehat{g} |^2 \big)^{1/2}  \big( \int^1_0 |\widehat{1}_{S} |^{6} \big)^{1/6}.\end{equation}
By Parseval's identity and the boundedness of $f_{\unf}, g, \chi$ we have
\begin{equation}\label{pars-bds} \int^1_0 |\widehat{f_{\unf} \chi} |^2  ,  \int^1_0 |\widehat{g} |^2  \ll N,\end{equation}
and Proposition \ref{ell-6} tells us that 
\[ \int^1_0 |\widehat{1}_S(t)|^6dt \ll N^2.\]
Finally, we note that 
\[ \widehat{f_{\unf}\chi}(t) = \int^1_0 \widehat{f_{\unf}}(t') \widehat{\chi}(t - t') dt' ,\] and so by property (4) of $\chi$ we have
\[ \Vert \widehat{f_{\unf} \chi} \Vert_{\infty} \leq \Vert \widehat{f_{\unf}} \Vert_{\infty} \Vert \widehat{\chi} \Vert_1 \ll_M N \mathcal{F}(M)^{-1}.\] Combining all these estimates together gives
\[ T_{\unf, \bullet'} = \sum_n (f_{\unf} \chi \ast g)(n) 1_S(n) \ll_M N^{3/2}\mathcal{F}(M)^{-1/3}.\] 
If the growth of $\mathcal{F}$ is sufficiently rapid, we obtain in view of the fact that $d, q \leq M$, $\eps = \delta/M$ and \eqref{eps-prime-below} that
\begin{equation}\label{unf-bullet} T_{\unf,\bullet'} \ll \delta (4\eps \eps')^{d+1} q^{-2} N^{3/2}.\end{equation}
An almost identical argument (relying instead on the bound $\Vert \chi' \Vert_1 = O_{M}(1)$) yields
\begin{equation}\label{unf-prime-bullet} T_{\bullet, \unf} \ll \delta (4\eps \eps')^{d+1} q^{-2} N^{3/2}.\end{equation}

Combining \eqref{tor-tor}, \eqref{bullet-sml}, \eqref{sml-tor}, \eqref{unf-bullet} and \eqref{unf-prime-bullet} with \eqref{tors} we obtain
\[ |S \cap Y_{-}|  (2 \eps')^{d+1}  q^{-1} N\ll \delta (4 \eps \eps')^{d+1} q^{-2} \mathfrak{S}(u + u', q)N^{3/2},\] and therefore
\[ |S \cap Y_-| \ll \delta  (2\eps)^{d+1} q^{-1}\mathfrak{S}(u + u', q) N^{1/2}.\]
Lemma \ref{lem36} provides the bound
\[ |S \cap (Y \setminus Y_{-})| \ll \delta  (2\eps)^{d+1} q^{-1}\mathfrak{S}(u + u', q) N^{1/2}.\]
Combining this with the preceding yields
\[ |S| \ll  \delta (2\eps)^{d+1} q^{-1} \mathfrak{S}(u + u', q)N^{1/2},\]
which is exactly \eqref{s-bound}. This completes the proof of Proposition \ref{sec-3-main}.

\section{The square-root of a Bohr set}
\label{sec5}

Suppose that $\N$ is partitioned into two colour classes $V$ and $W$, neither of which has a monochromatic solution to $x + y = z^2$. The main result of the last section, Proposition \ref{sec-3-main}, shows that if $V \cap [N, 2N)$ has size at least $N/2$ then $V + V$ contains almost all of the squares in a ``Bohr set'' $\Lambda := \{n \in \N : n \equiv b \md{q}, |\frac{n}{N} - x|,\Vert \theta n - z\Vert_{\T^d} \leq \eps\}$. This means that most of $\sqrt{\Lambda}$ must lie in $W$. In this section we examine the additive properties of such square roots $\sqrt{\Lambda}$. (Recall that $\sqrt{\Lambda}$ is by definition the set of \emph{integers} $n$ such that $n^2 \in \Lambda$.)


Here is the main result of the section.

\begin{proposition}\label{sec-4-main}
Let $\eta > 0$. Then there is a function $\Omega : \N^3 \rightarrow \R_+$ with the following property. Suppose we have $q, d \in \N$, $\eps > 0$, $x \in [0,3]$, $\theta, z \in \T^d$ and $N \in \N$. Suppose that $\theta$ is $(\Omega(q,d,1/\eps), N)$-irrational. Suppose that $b$ is a square modulo $q$ and set
\[ Y := \{ n \in \N : n \equiv b \md{q}, |\frac{n}{N} - x| , \Vert \theta n - z\Vert_{\T^d} \leq \eps\}.\]
Let $Y' \subset Y$ be a set containing all but at most $\eta (2\eps)^{d+1} q^{-1}\mathfrak{S}(b,q) N^{1/2}$ of the squares in $Y$. Then, for all but at most $O(\eta \eps q^{-1} N^{1/4})$ of the elements $t \in Q$, where 
\begin{equation}\label{Q-def} Q := \P( \big[ (2x)^{1/4} - \frac{\eps}{100}, (2x)^{1/4} + \frac{\eps}{100}\big]; N^{1/4}, q),  \end{equation}
we have $t^2 \in \sqrt{Y'} + \sqrt{Y'}$.
\end{proposition}
(Recall that $\P(I; N, q) := \{n \in \Z : n/N \in I, q | n\}$.) 

The proof of this is a little complicated so we break it down into a few lemmas. We have $\sqrt{Y} = \bigcup_{a \in \mathscr{A}}Z^a_+ \cup Z^a_-$, where
\begin{align}\nonumber Z^a_{\pm} := \{ n \in \N : n \equiv\pm a \md{q}, & (x - \eps)^{1/2} N^{1/2} \leq n \leq (x + \eps)^{1/2} N^{1/2}, \\ & \Vert \theta n^2 - z \Vert_{\T^d} \leq \eps\},\label{z-def}\end{align} and $\mathscr{A}$ is the set of square roots of $b$ in $\Z/q\Z$.
Define
\[ \tilde Z^a_{\pm} := \sqrt{Y'} \cap Z^a_{\pm};\]
then
\[ \sum_{a \in \mathscr{A}}|Z^a_{\pm} \setminus \tilde Z^a_{\pm}| \ll \eta  (2\eps)^{d+1} q^{-1}\mathfrak{S}(b,q) N^{1/2}, \] by assumption. It follows that there is some $a \in \mathscr{A}$ such that 
\begin{equation}\label{spill}
|Z^a_{\pm} \setminus \tilde Z^a_{\pm}| \ll \eta  (2\eps)^{d+1} q^{-1} N^{1/2}.
\end{equation} 
Henceforth, we fix this value of $a$ and write $Z_{\pm} = Z^a_{\pm}$ for brevity.
To orient ourselves we remark that, if $\Omega$ grows sufficiently rapidly then one could prove that
\[ |Z_{\pm}| \sim  (2\eps)^{d+1} q^{-1} N^{1/2}\] (here we are using $\sim$ somewhat informally). We will not need to explicitly prove any statement of this kind separately.

\begin{lemma}\label{lem5.2}
Suppose that $n_+ \in Z_+$. Then
\[ \# \{ n_- \in Z_- : n_- + n_+  = q^2 m^2 \; \mbox{for some $m \in \Z$} \} \ll (2\eps)^{d+1} q^{-1} N^{1/4} ,\] the implied constant being uniform in $n_+$ and independent of $a$ \textup{(}recall that $Z_{\pm}$ depends on $a$\textup{)}.
Similarly, if $n_-\in Z_-$ then
\[ \# \{ n_+ \in Z_+ : n_- + n_+  = q^2 m^2 \; \mbox{for some $m \in \Z$} \} \ll (2\eps)^{d+1} q^{-1} N^{1/4} ,\]  the implied constant being uniform in $n_-$ and in $a$.
\end{lemma}
\begin{proof}  The quantity we are interested in can be written as
\[ \sum_{m \in I(n_+)} 1_{\Vert \theta (q^2m^2 - n_+)^2 - z\Vert_{\T^d} \leq \eps},\]
where $I(n_+)$ is the interval \[ \frac{1}{q}\big( (x - \eps)^{1/2} N^{1/2} + n_+ \big)^{1/2} \leq m \leq \frac{1}{q}\big( (x + \eps)^{1/2} N^{1/2} + n_+ \big)^{1/2},\] the cardinality of which satisfies
\begin{equation}\label{I-length-bd} |I(n_+)| \ll \eps q^{-1} N^{1/4}\end{equation} uniformly in $n_+$.
To bound this above, take a majorant $\psi^+_{\eps}$ to the unit ball $B_{\eps}(0) \subset \T^d$, as in Lemma \ref{lem-a2}. Then our quantity is at most
\[ \sum_{m \in I(n_+)} \psi^+_{\eps} (\theta (q^2m^2 - n_+)^2 - z).\]
Fourier expanding $\psi^+_{\eps}$, this is
\[ \sum_{\mathbf{r} \in \Z^d} \widehat{\psi^+_{\eps}}(\mathbf{r})\sum_{m \in I(n_+)} e(q^4 \mathbf{r} \cdot \theta m^4 + \ldots),\]
where the dots denote terms of degree at most $2$ in $m$ (which can depend on $\mathbf{r}, n_+, \theta, z,q$). 
The contribution from $\mathbf{r} = 0$ is $|I(n_+)| (\int \psi^+_{\eps})$ which, by \eqref{I-length-bd} and property (1) of Lemma \ref{lem-a2}, is $\ll (2\eps)^{d+1} q^{-1} N^{1/4}$.
By Corollary \ref{cor-weyl-1} (and since $|I(n_+)| \leq N^{1/4}$), we have 
\[ \big|\sum_{m \in I(n_+)} e(q^4\mathbf{r} \cdot \theta m^4 + \ldots)\big| \leq N^{1/4}\big( \frac{q^4 \Vert \mathbf{r} \Vert_1}{\Omega(q,d,1/\eps)}\big)^{1/C_4}.\]

By Lemma \ref{lem-a2} (2), the contribution from $\mathbf{r} \neq 0$ is therefore
\begin{align*} & \ll N^{1/4} \big(\frac{q^4}{\Omega(q,d,1/\eps)}\big)^{1/C_4} \sum_{\mathbf{r} \in \Z^d \setminus \{0\}} |\widehat{\psi^+_{\eps}}(\mathbf{r})| \Vert \mathbf{r} \Vert_1 \\ & \ll_{\eps, d} N^{1/4} \big(\frac{q^4}{\Omega(q,d,1/\eps)}\big)^{1/C_4},\end{align*}
which is also $\ll (2\eps)^{d+1} q^{-1} N^{1/4}$ if $\Omega$ is chosen appropriately.
\end{proof}

Define progressions $P_+, P_-$ by
\begin{equation}\label{pdef} P_{\pm} := \{ n \in N : n \equiv \pm a \md{q}, (x - \eps)^{1/2} N^{1/2} \leq n \leq (x + \eps)^{1/2} N^{1/2}\},\end{equation} and recall from the statement of Proposition \ref{sec-4-main} the definition of $Q$, viz. 
\[ Q:= \P([(2x)^{1/4} - \frac{\eps}{100}, (2x)^{1/4} + \frac{\eps}{100}]; N^{1/4}, q).\]
Observe that if $t \in Q$ then $t^2$ is a sum $p_+ + p_{-}$ in $\gg \eps q^{-1} N^{1/2}$ ways. Indeed
\[ \big( (2x)^{1/2} - \frac{\eps}{10} \big) N^{1/2} < t^2 < \big( (2x)^{1/2} + \frac{\eps}{10} \big) N^{1/2} \] and $t^2 \equiv 0 \md{q}$, hence for any of the $\gg \eps q^{-1} N^{1/2}$ values of $p_+$ with $(x^{1/2} - \frac{\eps}{10}) N^{1/2} < p_+ < (x^{1/2} + \frac{\eps}{10}) N^{1/2}$ and $p_+ \equiv a \md{q}$ we have $t^2 - p_+ \in P_-$.

Note that from \eqref{z-def} and \eqref{pdef} we have
\begin{equation}\label{zp} Z_{\pm} = \{ n \in P_{\pm} : \Vert \theta n^2 - z \Vert_{\T^d} \leq \eps\}.  \end{equation}
This suggests the intuition behind the arguments that follow, which is that $Z_{\pm}$ behaves like a ``pseudorandom'' subset of $P_{\pm}$ of density $(2\eps)^{d}$. Thus it is reasonable to expect that a typical $t^2$, $t \in Q$, will have $\gg (2\eps)^{2d + 1} q^{-1} N^{1/2}$ representations as $z_+ + z_-$ with $z_+ \in Z_+$, $z_- \in Z_-$. 
\begin{lemma}\label{lem5.3}
Suppose that $\Omega$ grows sufficiently rapidly. Write $r(n)$ for the number of representations of $n$ as $z_+ + z_-$ with $z_{\pm} \in Z_{\pm}$.  Suppose that $\Omega$ grows fast enough. Then all but at most $\eta \eps q^{-1}N^{1/4}$ of elements $t \in Q$ have $r(t^2) \gg  (2\eps)^{2d+1} q^{-1}N^{1/2}$.
\end{lemma}
\begin{proof}
If the lemma is false then for any absolute constant $c$ (which we may specify later) there is a set $T \subset Q$, $|T| \geq \eta \eps q^{-1} N^{1/4}$, such that 
\begin{equation}\label{to-contradict} \sum_{t \in T} r(t^2) \leq c (2\eps)^{2d+1} q^{-1} |T|N^{1/2} .\end{equation}
We first introduce a smoothed variant of $r$, defined by
\[ \tilde r(n) = f_+ \ast f_- (n),\]
where 
\[ f_{\pm}(n) = 1_{P_{\pm}}(n) \psi^-_{\eps} (\theta n^2 - z),\] where $\psi^-_{\eps}$ is a suitable minorant to $B_{\eps}(0)$, as constructed in Lemma \ref{lem-a2}. From \eqref{zp} we see that $1_{Z_{\pm}} \geq f_{\pm}$ pointwise, and so
\[ r(n) \geq \tilde r(n)\] pointwise. 
Define
\[ g_{\pm}(n) = 1_{P_{\pm}}(n) (\psi^-_{\eps}(\theta n^2 - z) - \int \psi^-_{\eps}).\]
Fourier expanding $\psi^-_{\eps}$, we see that 
\[ \widehat{g_{\pm}}(t) = \sum_{\mathbf{r} \in \Z^d \setminus \{0\}} \widehat{\psi^-_{\eps}}(\mathbf{r}) \sum_{n \in P_+} e(\mathbf{r} \cdot \theta n^2 + nt - \mathbf{r} \cdot z).\]
Parametrising $n \in P_+$ as $n = qm + b$ for $m$ in some interval $I$ with $|I| = |P_+| < N^{1/2}$, it follows from Corollary \ref{cor-weyl-1} that the inner sum is $\ll N^{1/2}\Omega(q,d,1/\eps)^{-1/C_2} \Vert \mathbf{r} \Vert_1$. Therefore, by property (2) of Lemma \ref{lem-a2}, we have

\begin{align}\nonumber \Vert \widehat{g_{\pm}} \Vert_{\infty} & \ll  N^{1/2} \Omega(q,d,1/\eps)^{-1/C_2}\sum_{\mathbf{r} \in \Z^d} |\widehat{\psi^-_{\eps}}(\mathbf{r})| \Vert \mathbf{r} \Vert_1 \\ & \ll_{\eps, d} N^{1/2} \Omega(q,d,1/\eps)^{-1/C_2}.\label{g-unf}\end{align}
Now, writing
\[ f_{\pm} = 1_{P_{\pm}} \int \psi^-_{\eps} + g_{\pm},\] we may expand $\sum_{t \in T} \tilde r(t^2)$ as a sum of four terms. The ``main term'' is
\[ E_{\main} =  (\int \psi^-_{\eps})^2 \sum_{t \in T} 1_{P_+} \ast 1_{P_-}(t^2).\]
The three error terms each have the shape
\[ E_{\error}= \sum_{t \in T} g_{\pm} \ast h_{\mp}(t^2),\] where $h_{\mp}$ is bounded pointwise by 1 and supported on $P_{\mp}$. 

We have already remarked that if $t \in Q$ then $t^2$ has $\gg \eps q^{-1} N^{1/2}$ representations as $p_+ + p_-$, and therefore
\begin{equation}\label{e-main} E_{\main} \gg (2\eps)^{2d} \cdot |T| \cdot \eps q^{-1} N^{1/2} \gg \eta (2\eps)^{2d+2} q^{-2} N^{3/4} .\end{equation}
On the other hand
\[ E_{\error} = \int^1_0 \widehat{g_{\pm}}(\theta) \widehat{h_{\mp}}(\theta) \widehat{1_{T^2}}(\theta) d\theta,\] where here $T^2 := \{t^2 : t \in T\}$.
Using the same application of H\"older's inequality as in \eqref{holdered}, 
\[  E_{\error} \ll \Vert \widehat{g_{\pm}} \Vert_{\infty}^{1/3} \big(\int^1_0 |\widehat{g_{\pm}}|^2\big)^{1/3} \big(\int^1_0 |\widehat{h_{\mp}}|^2\big)^{1/2} \big(\int^1_0 |\widehat{1}_{T^2}|^6 \big)^{1/6}.\] 
By Parseval and the crude bound $|P_{\pm}| \ll  N^{1/2}$ we have
\[ \int^1_0 |\widehat{g_{\pm}}|^2,  \int^1_0 |\widehat{h_{\mp}}|^2 \ll N^{1/2}.\]
Proposition \ref{ell-6} tells us that 
\[ \int^1_0 |\widehat{1}_{T^2}|^6 \ll N .\] 
Putting this together with \eqref{g-unf} gives
\[ E_{\error} \ll  \Omega(q,d,1/\eps)^{-1/3C_2} N^{3/4}.\] Choosing $\Omega$ to grow sufficiently quickly, we see from \eqref{e-main} that this can be made less than $\frac{1}{10}$ of $E_{\main}$. It follows from \eqref{e-main} that 
\[ \sum_{t \in T} \tilde r(t^2) \geq E_{\main} - 3 E_{\error} > \frac{1}{2} E_{\main} \gg (2\eps)^{2d+1} q^{-1} |T|N^{1/2},\] contrary to \eqref{to-contradict} if $c$ was chosen small enough.
\end{proof}

Finally we put Lemmas \ref{lem5.2} and \ref{lem5.3} together to establish Proposition \ref{sec-4-main}. It is certainly enough (in view of the definitions of $\tilde Z_{\pm}$) to show that $\tilde Z_+ + \tilde Z_-$ contains $t^2$ for all but at most $O(\eta \eps q^{-1} N^{1/4})$ of the elements $t \in Q$. By Lemma \ref{lem5.3}, all but at most $\eta \eps q^{-1} N^{1/4}$ elements $t \in Q$ are such that $t^2$ is \emph{well-represented} in $Z_+ + Z_-$, by which we mean that $r(t^2) \gg (2\eps)^{2d+1} q^{-1}N^{1/2}$, where $r(t^2)$ is the number of representations of $t^2$ as $z_+ + z_-$. Suppose now that we pass from $Z_{\pm}$ to $\tilde Z_{\pm}$. The number of pairs $(z_+, z_-)$ with $z_+ + z_-$ the square of an element in $Q$ that are lost in this way is, by Lemma \ref{lem5.2}, bounded above by $\ll |Z_{\pm} \setminus \tilde Z_{\pm}| (2\eps)^{d+1} q^{-1}N^{1/4}$. By \eqref{spill}, this is bounded by $\ll \eta (2\eps)^{2d+2} q^{-2} N^{3/4}$. The number of $t$ for which $t^2$ is well-represented but does not lie in $\tilde Z_+ + \tilde Z_-$  is therefore bounded above by
\[ \ll \frac{\eta (2\eps)^{2d+2} q^{-2} N^{3/4}}{(2\eps)^{2d+1} q^{-1} N^{1/2}} = O(\eta \eps q^{-1}N^{1/4}).  \]
This completes the proof of Proposition \ref{sec-4-main}.

\section{Gaps between sums of two squares}\label{sec6}

In this section we prove a result, Proposition \ref{gaps-result}, that we will need in the next section. It seems possible that such a result appears in the literature already, but we do not know a reference. We prove a slightly more general result than we actually need since this is plausibly of independent interest.

\begin{proposition}\label{gaps-result}
Let $\alpha_1, \beta_1, \gamma_1, \alpha_2, \beta_2, \gamma_2$ be nonnegative reals with $\alpha_1 < \beta_1$, $\alpha_2 < \beta_2$, $\alpha_1^2 + \alpha_2^2 < \gamma_1 < \gamma_2 < \beta_1^2 + \beta_2^2$. Let $q \in \N$ and set $P_i := \P([\alpha_i, \beta_i]; N, q)$ for $i = 1,2$. Suppose that $\gamma_1 \leq n/N^2 \leq \gamma_2$.
Then there are $n_1 \in P_1$, $n_2 \in P_2$ such that 
\[ |n^2_1 + n^2_2 - n| \ll \sqrt{N}.\] The implied constant may depend on $\alpha_i, \beta_i, \gamma_i, q$ but is independent of $n$ and $N$.
\end{proposition}
\emph{Remark.} A well-studied case is that in which $P_1 = P_2 = \{1,\dots, N\}$. Then it is well-known that there is a sum of two squares $n_1^2 + n_2^2$ within $O(N^{1/2})$ of any $n \leq N^2$. One argument to prove this is very simple: take $n_1 = \lfloor \sqrt{n} \rfloor$, noting that $|n - n_1^2| \ll N$, and then set $n_2 := \lfloor \sqrt{n - n_1^2}\rfloor$. No bound of the form $o(N^{1/2})$ is known, a problem Montgomery \cite[Problem 64, p. 208]{montgomery} attributes to Littlewood. The argument just sketched does not adapt to our case since the $n_2$ produced is necessarily very small. However, there is another type of argument giving a similar bound and allowing us to take $n_1 \approx n_2$. The idea here is to take $n_1(k) = \lfloor \sqrt{n/2} \rfloor + k$, $n_2(k) = \lfloor \sqrt{n/2} \rfloor - k$, where $k \in Z$ is to be specified later.  Observe that 
\[ n_1(k)^2 + n_2(k)^2 = 2 \lfloor \sqrt{n/2} \rfloor^2 + 2k^2,\] and so in particular
\[ n_1(0)^2 + n_2(0)^2 \leq n,\]
\[ n_1(k)^2 + n_2(k)^2 \geq n - 2\sqrt{n} + 2k^2 > n\] for $k = \lceil \sqrt{n} \rceil$ and
\[ (n_1(k+1)^2 + n_2(k+1)^2) - (n_1(k)^2 - n_2(k))^2 = 4k + 2 \ll \sqrt{n}\] uniformly for $k \leq \lceil \sqrt{n}\rceil$. It follows from the ``discrete intermediate value theorem'' that there is some $k$ for which $|n_1(k)^2 + n_2(k)^2 - n| \ll \sqrt{n}$.

It turns out that this argument \emph{does} generalise to allow us to prove Proposition \ref{gaps-result}.

\begin{proof} For the duration of this proof, the implied constant in the $O()$ and $\ll, \gg$ notations may depend on $\alpha_i, \beta_i, \gamma_i, q$. We may clearly assume that $N$ is sufficiently large. 

For each $\gamma \in [\gamma_1, \gamma_2]$, define $I_{\gamma}$ to be the set of all $\lambda \in \R$ for which there exist $t_1, t_2 \in \R$ with $\alpha_1 \leq t_1 \leq \alpha_2$, $\beta_1 \leq t_2 \leq \beta_2$, $t_1/t_2 = \lambda$ and $t_1^2 + t_2^2 = \gamma$. Let $\tilde I_{\gamma}$ be the middle half of $I_{\gamma}$. It is easy to see that $I_{\gamma}$ is a closed interval whose length is positive and varies continuously as a function of $\gamma$, and is therefore bounded below uniformly in $\gamma$. The same is true for $\tilde I_{\gamma}$. This implies that 
\begin{enumerate}
\item There is an absolute $\eps \gg 1$ such that if $\lambda \in \tilde I_{\gamma}$ then we may find $t_1, t_2$ with $t_1/t_2 = \lambda$ and
\begin{equation}\label{space} \alpha_i + \eps \leq t_i \leq \beta_i - \eps;\end{equation}
\item $\tilde I_{\gamma}$ contains a rational $a(\gamma)/b(\gamma)$ with $a(\gamma), b(\gamma) = O(1)$ and neither $a(\gamma)$ nor $b(\gamma)$ zero.
\end{enumerate}
Now suppose that $n$ is given satisfying $\gamma_1 \leq n/N^2 \leq \gamma_2$. Set $\gamma := n/N^2$, and select rationals $a = a(\gamma)$, $b = b(\gamma)$, not both zero, as in (2) above. According to (1), there are $t_1, t_2$ with $t_1^2 + t_2^2 = \gamma$, $t_1/t_2 = a/b$ and such that \eqref{space} is satisfied.

Now set 
\[ n_1(k)  :=  q\lfloor \frac{t_1 N}{q} \rfloor + qkb, n_2(k) := q\lfloor \frac{t_2 N}{q}\rfloor - qka.\] 
Evidently $q | n_1(k) ,n_2(k)$. Moreover from \eqref{space} it follows that $\alpha_i \leq n_i(k)/N \leq \beta_i$ provided $|k| \leq cN$ for suitably small $c \gg 1$. Therefore for $k$ in this range we have $n_i(k) \in P_i$. 
Observe that 
\[ n_1(0)^2 + n_2(0)^2 \leq (t_1^2 + t_2^2)N^2 = n.\]
Also
\begin{align} \nonumber n_1(k)^2 & + n_2(k)^2 \\ & = \label{form} q^2 \big( \lfloor \frac{t_1 N}{q}\rfloor^2 + \lfloor \frac{t_2 N}{q}\rfloor^2 + 2k(a \{ \frac{t_2 N}{q}\} - b \{ \frac{t_1 N}{q}\}) + k^2(a^2 + b^2)\big) \\ & \geq n - O(N) - O(k) + q^2k^2(a^2 + b^2),\nonumber \end{align} and in particular
\[ n_1(k)^2 + n_2(k)^2 > n\] for some $k = O(\sqrt{N})$.

Moreover, from \eqref{form} again we have \[ 
\big| (n_1(k+1)^2 + n_2(k+1)^2) - (n_1(k)^2 - n_2(k))^2 \big|  = O(k).
\]
It follows from these properties and a discrete intermediate value argument that there is some $k = O(\sqrt{N})$ for which $|n_1(k)^2 + n_2(k)^2 - n| \ll \sqrt{N}$. The result follows.
\end{proof}

\section{Proof of the main theorem}\label{sec7}

In Proposition \ref{sec6-main} below we will synthesise the main results of Sections \ref{sec4} and \ref{sec5}, together with the following small (and well-known) lemma.

\begin{lemma}\label{simple}
Let $Q \subset \N$ be a finite arithmetic progression of size at least $100$, and suppose that $S \subset Q$ is a set of size at least $\frac{9}{10} |Q|$. Then $S + S$ contains a subprogression of $Q + Q$ of size at least $|Q|$ with the same common difference as $Q$.
\end{lemma}
\begin{proof}
By translating we may assume that $Q = \{1,\dots, m\}$. Suppose that $x \leq m$. Then the pairs $\{j, x - j\}$, $1 \leq j < x/2$, are disjoint. If $S + S$ does not contain $x$, then $S$ cannot contain both elements of any such pair, and hence $|Q \setminus S| \leq \lfloor x/2\rfloor$. Therefore $\lfloor x/2 \rfloor \leq \frac{m}{10}$, and so $x \leq\frac{m}{5} + 2$. A similar argument holds for $x \geq m$, with the conclusion now being that $2m - x \leq \frac{m}{5} + 2$. Thus $S + S$ contains the progression $\frac{m}{5} +2 < x < 2m - \frac{m}{5} - 2$. This is more than $m$ elements if $m \geq 100$. 
\end{proof}

\begin{proposition}\label{sec6-main}
Let $\eta > 0$. Suppose that $A \subset [N,2N)$ is a set of size at least $N/2$. Then $\sqrt{2\sqrt{2\sqrt{2A}}}$ contains a progression $\P(I; N^{1/8}, q)$ for some interval $I \subset [0.1, 10]$ with $|I| \gg 1$ and for some $q = O(1)$.
\end{proposition}
\begin{proof}
Let $\eta > 0$ be a quantity to be specified later. Let $\Omega : \N^3 \rightarrow \R_+$ be the growth function appearing in the statement of Proposition \ref{sec-4-main}. Apply Proposition \ref{sec-3-main} with this function. Let $q,d,\eps, \theta, z, b$ be as in the conclusion of that proposition. Taking $Y$ as in the statement of Proposition \ref{sec-4-main}, Proposition \ref{sec-3-main} then tells us that $Y' := (A + A) \cap Y = 2A \cap Y$ satisfies the hypotheses of Proposition \ref{sec-4-main}. It follows that $2\sqrt{Y'}$, and hence $2\sqrt{2A}$, contains $t^2$ for all but at most $O(\eta \eps q^{-1} N^{1/4})$ values of $t \in Q = \P([(2x)^{1/4} - \frac{\eps}{100}, (2x)^{1/4} + \frac{\eps}{100}]; N^{1/4}, q)$. Therefore $\sqrt{2\sqrt{2A}}$ contains all but at most $O(\eta \eps q^{-1} N^{1/4})$, and therefore at least $(1 - C\eta) |Q|$, of the elements of $Q$.  If $\eta$ is chosen suitably, this is at least $\frac{9}{10}|Q|$ elements of $Q$, and so by Lemma \ref{simple} we see that $2\sqrt{2\sqrt{2A}}$ contains a subprogression $Q' \subset Q$ of the form $Q' = \P(I; N^{1/4}, q)$ with $|I| \gg \eps$. Finally, note that $\sqrt{Q'}$ contains a progression of the form $\P(I'; N^{1/8}, q)$ for some $I' \subset [0.1, 10]$ with $|I| \gg \eps$.\end{proof}

We are finally ready to complete the proof of Theorem \ref{mainthm}. Suppose we have a 2-colouring $V \cup W$ of all sufficiently large positive integers, with no monochromatic solution to $x + y = z^2$. Without loss of generality, there are infinitely many $N$ such that $|V \cap [N, 2N)| \geq \frac{N}{2}$.  Then we have the following chain of inclusions:
\[ \sqrt{2V} \subset W,\]
\[ \sqrt{2\sqrt{2V}} \subset \sqrt{2W} \subset V,\]
\[ \sqrt{2 \sqrt{2 \sqrt{2V}}} \subset \sqrt{2V} \subset W.\]
It follows from Proposition \ref{sec6-main} that $W$ contains , for infinitely many $N$, a progression $\P(I_N; N^{1/8}, q_N)$, where $I_N \subset [0.1, 10]$, $|I_N| \gg 1$ and $q_N = O(1)$, both of these uniformly in $N$. By pigeonholing in the value of $q_N$, we may assume that $q_N = q$ does not depend on $N$. Moreover, taking $M = \lceil 10/\inf |I_N|\rceil$ we see that every $I_N$ contains one of the finite collection of intervals $[\frac{i}{M}, \frac{i+1}{M}]$, $M/10 \leq i \leq 10M$. Therefore we may pigeonhole in the choice of interval as well and assume that $I_N = I$ does not depend on $N$. Thus $W$ contains $\P(I; N^{1/8}, q)$ for some $I \subset [0,1, 10]$ and for infinitely many $N$. Rescaling $N$, we see that $W$ contains $\P([1, 1 + c]; N, q)$ for infinitely many $N$ and for some $c > 0$.

From now on, this is the only consequence of the elaborate techniques of the earlier parts of the paper that we will require. 

Using Proposition \ref{gaps-result} as a tool, we find longer and longer progressions inside $W$. The following lemma formalises this process.
\begin{lemma}\label{lem64}
Let $P_1 = \P([\alpha_1, \beta_1]; N, q)$ and $P_2= \P([\alpha_2, \beta_2], N, q)$. Suppose that $\gamma_1 > \sqrt{\alpha_1^2 + \alpha_2^2}$ and that $\gamma_2 < \sqrt{\beta_1^2 + \beta_2^2}$. Then if $N$ is large enough \textup{(}depending on $\alpha_i, \beta_i,\gamma_i, q$\textup{)} we have
\[ \P([\gamma_1, \gamma_2]; N, q) \subset \sqrt{ P_1^2 + P_2^{2} - P_1 - P_2}.\]
\end{lemma}
\emph{Remark.} Here and in what follows, $A^2$ means $\{a^2 : a \in A\}$ and \emph{not} $a \cdot a' : a, a' \in A$ as one might find in other literature.
\begin{proof}
Fix $\tilde \gamma_1$, $\tilde\gamma_2$ with $\gamma_1 >  \tilde\gamma_1 > \sqrt{\alpha_1^2 + \alpha_2^2}$ and $\gamma_2 < \tilde\gamma_2 <  \sqrt{\beta_1^2 + \beta_2^2}$. 
By Proposition \ref{gaps-result}, $P_1^2 + P_2^2$ has a point within $O(\sqrt{N})$ of every point of $\P([\tilde \gamma_1^2, \tilde \gamma_2^2]; N^2, q)$. $P_1 + P_2$ is a progression of length $\gg N$ consisting of multiples of $q$, and so it is easy to see that $P_1^2 + P_2^2 - P_1 - P_2$ contains all of $\P([\tilde \gamma_1^2, \tilde \gamma_2^2]; N^2, q)$ with the possible exception of points within $O(N)$ of the endpoints, and hence it contains $\P([\gamma_1, \gamma_2]; N^2, q)$.
\end{proof}

Starting from the fact that 
\begin{equation}\label{start-fact}\P([1, 1+c]; N, q) \subset W \qquad \mbox{for infinitely many $N$},\end{equation} we apply Lemma \ref{lem64} iteratively. Observe that if $n_1, n_2, n_3, n_4 \in W$ then $n_1^2 - n_3 \in V, n_2^2 - n_4 \in V$, and hence (if it is an integer) \[ \sqrt{n_1^2 + n_2^2 - n_3 - n_4} \in W.\] Thus if $P_1, P_2 \subset W$ then $\sqrt{ P_1^2 + P_2^{2} - P_1 - P_2} \subset W$. Using this observation and repeated applications of Lemma \ref{lem64}, we see that for any finite $k$ and any choice of closed intervals $I_i \subset (\sqrt{i}, (1 + c) \sqrt{i})$ there is an infinite sequence of $N$s such that $\P(I_i; N, q) \subset W$ for $i = 1,2,\dots, k$. 

We claim that there is some $k = k(c)$ and some choice of $I_1,\dots, I_k$ such that $\bigcup_{i = 1}^k I_i$ contains an interval of the form $[x, 3x]$. First note that if $i > 1/2c$ then $(1 + c)\sqrt{i} > \sqrt{i+1}$, and so the intervals $(\sqrt{i}, (1 + c) \sqrt{i})$ and $(\sqrt{i+1}, (1 + c)\sqrt{i+1})$ overlap. Thus if we set $i_0 := \lceil 1/2c\rceil$ and $i_1 := 9i_0$ then $\bigcup_{i_0 \leq i \leq i_1} (\sqrt{i}, (1 + c) \sqrt{i})$ is an interval containing a subinterval of the form $[x,3x]$. 

Thus $W$ contains $\P([x, 3x]; N, q)$ for infinitely many $N$, and hence (replacing $N$ by $\lfloor 1.1 x N\rfloor$) we see that we have bootstrapped \eqref{start-fact} to the stronger statement that 
\[ \P([1,2]; N, q) \subset W \qquad \mbox{for infinitely many $N$}.\]
Pick one such $N = N_0$, sufficiently large. Thus
\begin{equation}\label{use-soon-1}  \P([1,2]; N_0, q) \subset W.\end{equation}
By Lemma \ref{lem64} once more (and the inequalities $\sqrt{2} < \frac{3}{2} < \frac{5}{2} < \sqrt{8}$) we have
\[ \P([\frac{3}{2}, \frac{5}{2}]; N_0, q) \subset W.\] Together with \eqref{use-soon-1}, this implies that 
\[ \P([1,2]; N_0 + 1, q) \subset W.\]
Continuing inductively, we obtain
\[ \bigcup_{N \geq N_0}  \P([1,2]; N, q) \subset W.\]
This implies that all sufficiently large multiples of $q$ lie in $W$. But there are arbitrarily large multiples $x,y,z$ of $q$ satisfying $x + y = z^2$, and so at last we obtain a contradiction.

\appendix

\section{Some smooth cutoff functions}\label{bump-appendix}

In the main body of the paper we required various smooth cutoff functions to (characteristic functions of) discrete intervals, balls in the torus $\T^d$ and Bohr sets. In this appendix we prove the existence of functions with required properties.

It is convenient to have a $C^{\infty}$-function $f : \R \to [0,1]$ with $\Supp(f)\subset [-1,1]$ and $\int f(x) dx = 1$. Such a function can be constructed with a ``trick'', for example defining $f(x) = C \exp (\frac{1}{x^2 - 1})$ for an appropriate constant $C$ (for a very elegant analysis of this, see \cite[Lemma 9]{bombieri-friedlander-iwaniec}), or by convolving an infinite sequence of normalised characteristic functions of intervals $[-\ell_j, \ell_j]$ with $\sum_j \ell_j \leq 1$.

Let $g : \R \rightarrow \R$ be any compactly supported $C^{\infty}$ function (for example, $f$). Then, since the $M$th derivative $g^{(M)}$ is continuous and supported on $[-1,1]$, we have the bound $\Vert g^{(M)} \Vert_{\infty} = O_M(1)$. By integration by parts this leads to the standard bound \begin{equation}\label{standard} |\widehat{g}(\xi)| \ll_M \min( 1, |\xi|^{-M})\end{equation} for $\xi \in \R$, where here $\widehat{g}(\xi) = \int_{\R} g(x) e(- \xi x)dx$.

\begin{lemma}\label{app-1}
Let $N \in \N$. There is a function $\psi = \psi_N : \N \rightarrow [0,\infty)$ with $\psi(n) = 1$ for $N \leq n < 2N$ and $\Vert \widehat{\psi} \Vert_1 = O(1)$ \textup{(}uniformly in $N$\textup{)}, where the Fourier transform $\widehat{\psi}(t)$ is defined to be $\sum_n \psi(n) e(-tn)$ for $t \in \T$.
\end{lemma}
\begin{proof} (Sketch.)
Define first a function $g : \R \rightarrow \R$ via $g = 1_{[0,3]} \ast f$. It is easy to check that $g$ is $C^{\infty}$, compactly supported, and that $g(x) = 1$ for $x \in [1,2]$. We may then define $\psi(n) := g(n/N)$. By the Poisson summation formula we have 
\[ \widehat{\psi}(\theta) = N \sum_{k \in \Z} \widehat{g}(N (k + \theta)),\]
and so
\[ \Vert \widehat{\psi} \Vert_1 \leq N \int^{\infty}_{\infty} |\widehat{g}(Nu)| du  = \Vert \widehat{g} \Vert_1,\]
where the $\ell^1$ norm on the right is taken on $\R$. The bound $\Vert \widehat{g} \Vert_1 = O(1)$ follows quickly by taking $M = 2$ in \eqref{standard}.

Alternatively, one may take $\psi$ to be a de la Vall\'ee Poussin type kernel as in the figure and proceed quite explicitly using the fact that this is a difference of two Fej\'er kernels. Details may be found in \cite[Section 1.2]{katznelson}.
\begin{figure}[htpb]
\centering
\begin{tikzpicture}
\draw[->] (0,0) -- (3.5,0);
\draw[->] (0,0) -- (0,1.5);
\draw[line width=2] (0,0)--(1,1)--(2,1)--(3,0);
\draw (1,0.2)--(1,-0.2) node[below] {$N$};
\draw (2,0.2)--(2,-0.2) node[below] {$2N$};
\draw (3,0.2)--(3,-0.2) node[below] {$3N$};
\draw (0.2,1)--(-0.2,1) node[left] {$1$};
\end{tikzpicture}
\caption{de la Vall\'ee Poussin kernel.}
\label{fig-1}
\end{figure}
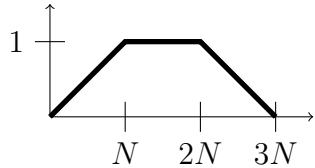
\end{proof}

Suppose now that $\eps > 0$ and that $d \in \N$.
Let us define $f_\eps : \T^d\to [0,\infty)$ by $f_\eps(x) = (2\eps)^{-d} \prod_{i=1}^df(\tilde x_i/\eps)$, where $\tilde x$ is the unique element of $(-\frac{1}{2},\frac{1}{2}]^d$ mapping to $x$ under the natural projection. Note that $\int_{\T^d} f_\eps(x)dx = 1$.

\begin{lemma}\label{lem-a2}
There is a majorant $\psi^+_{\eps}$ and a minorant $\psi^-_{\eps}$ to the ball $B_{\eps}(0)$ in $\T^d$ satisfying
\begin{enumerate}
\item $\frac{1}{2} \leq (2\eps)^d\int_{\T^d} \psi^{\pm}_{\eps}(t) dt \leq 2$ and
\item $\sum_{\mathbf{r} \in \Z^d \setminus \{0\}} |\widehat{\psi^{\pm}_{\eps}}(\mathbf{r})| \Vert \mathbf{r} \Vert_{1} =  O_{\eps, d}(1)$.
\end{enumerate}
\end{lemma}
\begin{proof}
We construct $\psi_{\eps}^+$. The construction of $\psi^-_{\eps}$ is very similar and is left to the reader.
Set $\eps' := \eps/10d$. For $x\in \T^d$ set
\[
    \psi^+_\epsilon(x) = 1_{B_{\epsilon + \eps'}(0)} * f_{\epsilon'}(x) = \int_{\T^d} f_{\epsilon'}(x-y)1_{B_{\epsilon + \eps'}(0)}(y)dy. 
\]
Since $f_{\eps'}$ is supported on $B_{\eps'}(0)$, $\psi^+_{\eps}(x) = 1$ for $x \in B_{\eps}(0)$, and in particular $\psi^+_{\eps}$ is a majorant to the ball $B_{\eps}(0)$.

Moreover $\psi_{\eps}$ is bounded pointwise by $1$ and is supported on $B_{\eps + \eps'}(0)$, whence 
\[ \int_{\T^d} \psi^+_{\eps}(t) dt \leq \mu_{\T^d}(B_{\eps + \eps'}(0)) = (1 + \frac{\eps'}{\eps})^d (2\eps)^d \leq 2(2\eps)^d.\]
Thus (1) is satisfied.

Next we turn to point (2). Suppose that $\mathbf{r} \in \Z^d \setminus \{0\}$. Write $\mathbf{r} = (r_1,\dots, r_d)$, and assume without loss of generality that $|r_1| = \|\mathbf{r}\|_\infty$. Performing $M$ integration by parts in the integral
\[
    \widehat{\psi^+_\eps}(\mathbf{r}) = \int_{\T^d} \psi^+_\eps(x)e(-x\cdot\mathbf{r}) dx
\]
with respect to $x_1$, to get that 
\[
    \widehat{\psi^+_\eps}(\mathbf{r}) = \frac{1}{(-2\pi i r_1)^M}\int \frac{\partial^M \psi^+_\eps(x)}{\partial x_1^M}e(-x\cdot\mathbf{r})dx \ll_{\epsilon,d,M} \|\mathbf{r}\|_\infty^{-M}
\]
for any $M \in \N$ (this is essentially the same bound as  \eqref{standard}). The $\ell^1$ and $\ell^{\infty}$ norms of $\mathbf{r}$ are comparable up to factors of $O_d(1)$, and hence
\[ \sum_{\mathbf{r}\in \Z^d \setminus \{0\}} |\widehat{\psi^+_\eps}(\mathbf{r})| \|\mathbf{r}\|_1\ll_{\eps,d,M} \sum_{\mathbf{r} \in \Z^d \setminus \{0\}} \Vert \mathbf{r} \Vert_1^{1 - M}.\] Taking $M = d+2$, it is easy to see that the sum on the right converges and is bounded by $O_d(1)$.
\end{proof}

Finally we turn to the most complicated of our constructions, a smooth approximant for the Bohr-type set $X$ considered in Section \ref{sec5}.

\begin{lemma}\label{lem-a3}
Let $0 < \eps' < \eps < 1$, $d , q \in \N$, $x \in \R$ and $\theta, z \in \T^d$. Then there is an $A = A(\eps, \eps', d, q)$ with the following property. Suppose that $N$ is sufficiently large in terms of $\eps, \eps', d, q, A$.  Set
\[
X = \{ n \in \N : n \equiv u \md{q}, |\frac{n}{N} - x|, \|\theta n - z\|_{\T^d} \leq \eps \}
\]
and
\[
X_{-} = \{ n \in \N : n \equiv u \md{q}, |\frac{n}{N} - x|, \|\theta n - z\|_{\T^d} \leq \eps-\eps' \}.
\]
Suppose that $\eps' < \eps/10d$ and $\theta$ is $(A,N)$-irrational. Then there exists a function $\chi$ satisfying
\begin{enumerate}
\item $1_{X_{-}}(n)\leq\chi(n) \leq 1_X(n)$ for all $n$;
\item $\Vert \widehat{\chi} \Vert_1 = O_{\eps,\eps',q,d}(1)$ and
\item $\sum_n \chi(n) \geq \frac{1}{2}(2\eps)^{d+1}q^{-1} N$.
\end{enumerate}

\end{lemma}
\begin{proof}
Let $g : \R \rightarrow [0,\infty)$ be a $C^{\infty}$ function with $g(t) = 1$ for $|t - x| \leq \eps - \eps'$ and $g(t) = 0$ for $|t - x| > \eps$. Such a function can be obtained by convolving the characteristic function of the interval $\{t : |t - x| \leq \eps - \frac{1}{2}\eps'\}$ with the function $\frac{2}{\eps'} f( \frac{2t}{\eps'})$.

Define a function $h : \T^d \rightarrow [0,\infty)$ by 
\[ h := f_{\eps'/2}*1_{B_{\eps-\eps'/2}(z)} .\]

Now define
\[ \chi(n) := g(\frac{n}{N}) h(\theta n) 1_{n \equiv u \md{q}}.\] 
The relevant support properties (1) may be easily checked. 
Turning to point (2), we begin by noting the expansion \[ 1_{n \equiv u \mdsub{q}} = q^{-1}\sum_{s \in \Z/q\Z} e(\frac{(n - u) s}{q}).\] This implies that 
\begin{equation}\label{eq33} \widehat{\chi}(t) = q^{-1} \sum_{s \in \Z/q\Z} e(-\frac{us}{q}) \widehat{g(\frac{\cdot}{N}) h(\theta \cdot)} (t + \frac{s}{q}).\end{equation} Therefore in order to establish (2) is suffices to prove that 
\begin{equation}\label{suffice-for-2} \Vert \widehat{g(\frac{\cdot}{N}) h(\theta \cdot)} \Vert_1 = O_{\eps, \eps', d}(1). \end{equation}
Fourier expanding $h$ and applying Poisson summation, we have
\begin{align}\nonumber
\widehat{g(\frac{\cdot}{N}) h(\theta \cdot)}(t) & = \sum_n g(\frac{n}{N}) h(\theta n) e(-tn) \\ \nonumber & = \sum_n g(\frac{n}{N}) \sum_{\mathbf{r}} \widehat{h}(\mathbf{r}) e\big((\mathbf{r} \cdot \theta - t) n \big) \\ & = N \sum_{\mathbf{r}} \widehat{h}(\mathbf{r}) \sum_{k \in \Z} \widehat{g}\big( N (t + k - \mathbf{r} \cdot \theta) \big).\label{gh-exp}
\end{align}
Thus
\[
\Vert \widehat{g(\frac{\cdot}{N}) h(\theta \cdot)} \Vert_1  \leq N \sum_{\mathbf{r}} |\widehat{h}(\mathbf{r})| \int^{\infty}_{-\infty} |\widehat{g}(Nu)| du  = \Vert \widehat{g} \Vert_1 \Vert \widehat{h} \Vert_1,
\]
where here the $\ell^1$ norms are on $\Z^d$ and $\R$ respectively.

That $\Vert \widehat{g} \Vert_1 \ll_{\eps, \eps'} 1$ follows immediately from \eqref{standard} with $M = 2$.

By essentially the same reasoning used in the proof of Lemma \ref{lem-a2} we have \begin{equation}\label{moreover} |\widehat{h}(\mathbf{r})| \ll_{\eps, \eps', d, M} \Vert \mathbf{r} \Vert_{\infty}^{-M}.\end{equation} 
Taking $M = d + 1$ we obtain
\[ \Vert \widehat{h} \Vert_1 = O_{\eps, \eps', d}(1).\]
Putting these facts together completes the proof of \eqref{suffice-for-2} and hence of (2).

It remains to verify (3). Note that we have not yet used the irrationality of $\theta$. From \eqref{eq33} we have
\[ \sum_n \chi(n)  = \widehat{\chi}(0)  = q^{-1} \sum_{s \in \Z/q\Z} e(-\frac{us}{q}) \widehat{g(\frac{\cdot}{N}) h(\theta \cdot)}(\frac{s}{q}) .\]
By \eqref{gh-exp}, it follows that 
\begin{equation}\label{chi-exp} \sum_n \chi(n) = N q^{-1} \sum_{\mathbf{r} \in \Z^d} \sum_{s \in \Z/q\Z} \sum_{k \in \Z} e(-\frac{us}{q}) \widehat{h}(\mathbf{r}) \widehat{g}\big(N (\frac{s}{q} + k - \mathbf{r} \cdot \theta) \big).\end{equation}

The contribution from $\mathbf{r} = 0$, $s = 0$, $k = 0$ is $N q^{-1} (\int_{\T^d} h) (\int_{\R} g)$. Since $\eps' < \eps/10d$ we have $\int_{\T^d} h \geq \mu_{\T^d}(B_{\eps - \eps'}(0)) \geq 0.9 (2\eps)^d$, and evidently $\int_\R g \geq 2(\eps - \eps') > 0.9 (2\eps)$. Thus the contribution from this term is $\geq \frac{3}{4}(2\eps)^{d+1} q^{-1} N$. To complete the proof of (3) it suffices to show that the contribution of the other terms to \eqref{chi-exp} is at most $\frac{1}{4}(2\eps)^{d+1} q^{-1} N$, to which end it is enough to show that 
\begin{equation}\label{ets} \sum_{\mathbf{r} \in \Z^d} \sum_{s \in \Z/q\Z} \sum_{k \in \Z} |\widehat{h}(\mathbf{r})| |\widehat{g}(N (\frac{s}{q} + k - \mathbf{r} \cdot \theta) | \leq \frac{1}{4} (2\eps)^{d+1},\end{equation} where the sum omits the term $\mathbf{r} = 0$, $s = 0$, $k = 0$.

By \eqref{moreover} (with $M = d+ 1$) and \eqref{standard} (with $M = 2$), the left hand side is bounded by 
\begin{equation}\label{ets2} O_{\eps, \eps',d}(1) \sum_{\mathbf{r} \in \Z^d} \sum_{s \in \Z/q\Z} \sum_{k \in \Z} \min(1,  \Vert \mathbf{r} \Vert^{-d-1}) \min(1, N^{-2}| k + \frac{s}{q} - \mathbf{r} \cdot \theta |^{-2} ).\end{equation}
If $0 < \Vert \mathbf{r} \Vert_1 \leq A/q$ then it follows from the fact that $\theta$ is $(A,N)$-irrational that $| k + \frac{s}{q} - \mathbf{r} \theta  | \geq \frac{A}{qN}$ (no matter the value of $s$ or $k$). The same is trivially true when $\mathbf{r} = 0$, provided that not both of $s, k$ are zero and that $N$ is sufficiently large.
In the inner sum over $k$ in \eqref{ets2}, the contribution from all but at most one term is $\ll N^{-2}\sum_{m \in \Z \setminus \{0\}} |m|^{-2} \ll N^{-2}$, and so when $\Vert \mathbf{r} \Vert_1 \leq A/q$ the inner sum over $k$ is $\ll \frac{ q^2}{A^2} + N^{-2}$, which is $\ll q^2/A^2$ if $N$ is big enough. Therefore
\begin{align*} \sum_{\substack{\mathbf{r} \in \Z^d \\ \Vert \mathbf{r} \Vert \leq A/q }} & \sum_{s \in \Z/q\Z} \sum_{k \in \Z}  \min(1, \Vert \mathbf{r} \Vert^{-d-1}) \min(1, N^{-2}| k + \frac{s}{q} - \mathbf{r} \cdot \theta |^{-2} ) \\ &  \ll \frac{q^3}{A^2} \sum_{\mathbf{r}} \Vert \mathbf{r} \Vert^{-d-1} \ll_{d,q} A^{-2}.\end{align*}

All other terms in \eqref{ets2} have $\Vert \mathbf{r} \Vert \geq \frac{A}{q}$. Using the trivial bound
\[ \sum_{k \in \Z} \min(1, N^{-2} |k + \frac{s}{q} - \mathbf{r}\cdot \theta|^{-2}) \ll 1,\]
the contribution from these is bounded by
\[ O_{d,\eps, \eps',q}(1) \sum_{\Vert \mathbf{r} \Vert \geq A/q} \Vert \mathbf{r} \Vert^{-d-1} \ll_{d,\eps,\eps', q} A^{-1} .\]
Putting all of this together shows that \eqref{ets2} is bounded by $O_{d,\eps, \eps, q}(A^{-1})$, and so \eqref{ets} does indeed hold if $A$ is large enough as a function of $\eps, \eps', d, q$.
\end{proof}

\end{document}